\documentclass[11pt]{article}

\usepackage[a4paper,margin=1in]{geometry}
\usepackage[T1]{fontenc}
\usepackage[utf8]{inputenc}
\usepackage{lmodern}
\usepackage{microtype}
\usepackage{amsmath,amssymb,amsfonts,amsthm}
\usepackage{graphicx}
\usepackage{booktabs,makecell,array,tabularx}
\usepackage{adjustbox}
\usepackage{flafter}
\usepackage{placeins}
\usepackage[numbers,sort&compress]{natbib}
\usepackage{xcolor}
\usepackage[colorlinks=true,linkcolor=blue,citecolor=blue,urlcolor=blue]{hyperref}
\providecommand{\Eproj}{}

\newtheorem{theorem}{Theorem}[section]
\newtheorem{proposition}[theorem]{Proposition}
\newtheorem{assumption}[theorem]{Assumption}
\newtheorem{remark}[theorem]{Remark}

\newcommand{\A}{\mathcal A}
\newcommand{\nn}{\mathbf n}
\newcommand{\tg}{\mathbf t}
\newcommand{\curl}{\operatorname{curl}}
\newcommand{\HoneG}{H^1_\Gamma(\Omega)}
\newcommand{\Honeg}{H^1_g(\Omega)}
\newcommand{\ip}[2]{\left\langle #1,#2\right\rangle}
\newcommand{\Cond}{\operatorname{cond}}
\newcommand{\Eapp}{\mathcal E_{\rm app}}
\renewcommand{\Eproj}{\mathcal E_{\rm proj}^{\rm up}}

\makeatletter
\@ifpackageloaded{flafter}{}{\usepackage{flafter}}
\@ifpackageloaded{adjustbox}{}{\usepackage{adjustbox}}
\@ifpackageloaded{tabularx}{}{\usepackage{tabularx}}
\@ifpackageloaded{placeins}{}{\usepackage{placeins}}
\@ifundefined{proof}{\newenvironment{proof}[1][Proof]{\par\noindent\textit{#1.}\ }{\hfill$\square$\par}}{}
\@ifundefined{widetable}{\newenvironment{widetable}{\begin{adjustbox}{max width=\textwidth}}{\end{adjustbox}}}{}
\makeatother

\title{A Natural Decomposition Method for Essential Boundary Conditions in Noninterpolatory Meshfree Spaces}
\author{%
	Jingkai Zhang\textsuperscript{1}
	\and Tiexiang Li\textsuperscript{1,2,*}
	\and Shuo Zhang\textsuperscript{3,4}
}
\date{}
\providecommand{\tablenote}[1]{%
	\par\smallskip
	\begin{minipage}{0.96\textwidth}
		\footnotesize
		\textit{Note.} #1
	\end{minipage}
}


\begin{document}
	\raggedbottom
	
	\maketitle
	
	\begin{center}
		\small
		\textsuperscript{1}School of Mathematics and Shing-Tung Yau Center, Southeast University, Nanjing 210096, China\\
		\textsuperscript{2}Shanghai Institute for Mathematics and Interdisciplinary Sciences (SIMIS), Shanghai 200433, China\\
		\textsuperscript{3}State Key Laboratory of Mathematical Sciences (SKLMS) and State Key Laboratory of Scientific and Engineering Computing (LSEC), Institute of Computational Mathematics and Scientific/Engineering Computing, Academy of Mathematics and Systems Science, Chinese Academy of Sciences, Beijing 100190, China\\
		\textsuperscript{4}School of Mathematical Sciences, University of Chinese Academy of Sciences, Beijing 100049, China\\
		\textsuperscript{*}Corresponding author. Email addresses: \href{mailto:jkzhangmath@seu.edu.cn}{jkzhangmath@seu.edu.cn},
		\href{mailto:txli@seu.edu.cn}{txli@seu.edu.cn},
		\href{mailto:szhang@lsec.cc.ac.cn}{szhang@lsec.cc.ac.cn}
	\end{center}
	
	\begin{abstract}
		This paper develops a natural decomposition method (NDM) for imposing essential
		boundary conditions in noninterpolatory meshfree Galerkin spaces without
		boundary parameter tuning or auxiliary constraint construction. In such spaces,
		algebraic coefficients generally do not coincide with boundary values; hence
		coefficient assignment or nodal boundary prescription is not equivalent to
		imposing the continuous trace required by the variational problem. NDM introduces
		boundary data before discretization through a natural transfer mechanism: a
		source subproblem accounts for the forcing term, a weighted curl correction transfers
		the remaining trace mismatch, and a scalar recovery step reconstructs the
		solution from the corrected weighted gradient. For topologically trivial single
		domains with connected boundary, the reconstructed solution is equivalent, at
		the continuous level, to the solution satisfying the prescribed essential boundary data.
		The discrete analysis separates the approximation defect of the recovery space
		from the upstream transfer error visible to that space. Numerical experiments on
		benchmark problems evaluate the proposed transfer
		mechanism and report the associated
		conditioning, computational cost, and boundary perturbation behavior.
	\end{abstract}
	
	
	\bigskip

	\section{Introduction}\label{sec:intro}
	
	\subsection{Background and Motivation}
	
	Meshfree discretizations approximate the solution from scattered centers or point clouds rather than from a prescribed element mesh. This geometric freedom is attractive in computations involving large deformation, crack growth, moving boundaries, complex domains, and multiphase or interface configurations. Moving least squares, reproducing kernel approximations, radial basis functions (RBFs), and related constructions have accordingly become established tools for partial differential equations~\citep{Babuska2003,Belytschko1996,ChenWangDong2004,LiLuHanLiuSimkins2004,LiuHanLuLiCao2004,Nguyen2008}. For weak form meshfree discretizations, however, the imposition of essential boundary conditions remains a persistent obstacle. The difficulty is most pronounced for noninterpolatory spaces, where the algebraic coefficients are not nodal values and therefore do not directly encode the boundary trace~\citep{ChenWang2000,FernandezMendezHuerta2004,GuntherLiu1998,HillmanLin2021}.
	
	The obstruction is structural. In MLS, RKPM, and many Galerkin meshfree approximations, the basis functions do not generally satisfy the Kronecker $\delta$ property. Prescribing coefficients, or enforcing values at a finite set of boundary points, is therefore not equivalent to prescribing the trace of the trial function as in a conforming finite element space~\citep{ChenWang2000,FernandezMendezHuerta2004,GuntherLiu1998}. This distinction enters the variational formulation directly, because the problem with essential boundary conditions is posed on $H^1_g(\Omega)$ with test functions in $H^1_0(\Omega)$. Hillman and Lin~\citep{HillmanLin2021} showed that nodal enforcement on the boundary need not generate trial and test spaces with the required trace properties; in the absence of a weak Kronecker $\delta$ property, Galerkin orthogonality and best approximation may be lost. Thus an essential boundary treatment must address more than the size of a boundary residual. It must also determine how trace data are made compatible with the approximation space and with the interior weak form.
	
	One approach is to alter the approximation space near the boundary so that interpolation, weak Kronecker behavior, or boundary admissibility is recovered~\citep{HuertaFernandezMendezLiu2004,KrongauzBelytschko1996}. Admissible approximations, reproducing kernels with nodal interpolation properties, almost everywhere partition of unity constructions, and conforming window functions are representative examples~\citep{ChenEtAl2003,GoszLiu1996,KoesterChen2019,OhJeong2009}. These methods attack the mismatch at the level of the trial space and can support strong, or approximately strong, boundary enforcement. Their practical cost is the additional design required near the boundary. Local geometry, support selection, node distribution, and interface layout may all affect the construction, especially for nonconvex domains or geometrically complicated interfaces~\citep{KoesterChen2019,OhJeong2009}.
	
	A second approach leaves the space unchanged and modifies the weak formulation or adds auxiliary constraints. Lagrange multiplier, penalty, Nitsche, and consistent weak form corrections fall into this class~\citep{BelytschkoLuGu1994,FernandezMendezHuerta2004,HillmanLin2021,Nitsche1971,ZhuAtluri1998}. Lagrange multipliers introduce a saddle point problem whose stability depends on the primal and multiplier spaces~\citep{Babuska1973,BelytschkoLuGu1994,Brezzi1974,HuertaFernandezMendez2000}. Penalty methods are simple, but a small parameter enforces the trace weakly and a large parameter can lead to severe conditioning problems~\citep{ChoSongChoi2008,ZhuAtluri1998}. Nitsche formulations retain consistency without multiplier unknowns, yet require stabilization. In meshfree spaces this stabilization can depend on node distribution, support radius, approximation order, kernel metric, local geometry, and material coefficients~\citep{CostaPimentaWriggers2016,FernandezMendezHuerta2004,GriebelSchweitzer2003,HillmanLin2021}. Spatially varying stabilization, local generalized eigenvalue estimates, and variational multiscale boundary treatments reduce reliance on global empirical parameters but retain method-dependent boundary choices~\citep{GroeneveldHillman2024,Hughes1995,JimenezRecioSchweitzer2024}. Thus a boundary treatment should be judged not only by trace residuals, but also by parameter selection, auxiliary unknowns, conditioning, and the transfer of boundary data into the interior solution.
	
	The issue just described is specific to weak form noninterpolatory discretizations. In RBF collocation, RBF finite difference schemes, and related strong form methods, boundary conditions can often be inserted directly into the algebraic equations. Such methods have their own stability issues, including biased stencils near the boundary, irregular boundary nodes, quadrature or weight choices, and shape parameters~\citep{FornbergFlyer2015,LeBorneLeinen2023,NetuzhylovEtAl2007}. They do not, however, face the same variational admissibility question. The present work is concerned with the weak form setting, where consistency is tied to the trace properties of the trial and test spaces.
	
	These observations motivate the numerical design used below. Fern\'andez-M\'endez and Huerta~\citep{FernandezMendezHuerta2004} compared multiplier, penalty, Nitsche, and finite element coupling strategies in a Galerkin meshfree framework, showing that boundary treatment affects accuracy, residual control, and conditioning. Hillman and Lin~\citep{HillmanLin2021} clarified that boundary nodal enforcement does not generally construct the trace admissible trial space required by the variational problem, while Groeneveld and Hillman~\citep{GroeneveldHillman2024} developed a meshfree variational multiscale treatment to reduce reliance on a global penalty parameter. The numerical experiments below therefore compare boundary treatment mechanisms in the same noninterpolatory MQ RBF space. In this setting, NDM converts the essential boundary constraint into a source solve, a tangential curl correction, and a scalar recovery step before the Galerkin discretization is applied.

	\paragraph{Notation.}
	The notation is fixed as follows. The symbol $(\cdot,\cdot)_D$ denotes the $L^2(D)$ inner product, and $\langle\cdot,\cdot\rangle_\Sigma$ denotes the duality pairing on a boundary or interface $\Sigma$. The outward unit normal on $\Gamma$ is $\nn$, and $\nn_0$ is the interface normal directed from $\Omega_1$ to $\Omega_2$. Throughout the paper, $\A$ denotes the symmetric positive definite square root of the physical diffusion tensor, so the tensor in the elliptic operator is $\A^2=\A^\top\A$. The symbols $\kappa_1$ and $\kappa_2$ are reserved for the interface solution and flux jumps.
	
	\subsection{Overview and Outline}
	\label{subsec:contributions_outline}
	
	The construction is motivated by the planar natural decomposition of~\citep{YuZhang2025}. In two dimensions, the boundary mismatch has one tangential component and the correction field can be generated by a scalar potential. \hyperref[app:2d]{Appendix~\ref*{app:2d}} recalls this planar formulation. Here the planar theory serves as the prototype for the three dimensional construction. The essential three dimensional change is the replacement of the scalar correction by an $H(\curl)$ vector potential, because a surface or interface mismatch has two independent tangential components.
	
	The paper is organized around the source, curl, and recovery mechanism
	\[
	\begin{gathered}
		\text{Essential boundary data }(f,g)
		\Longrightarrow
		\text{source field }\widetilde u
		\Longrightarrow
		\text{curl correction }\A^{-1}\curl\boldsymbol\phi\\
		\Longrightarrow
		\text{recovered scalar }u^\star .
	\end{gathered}
	\]
	
	The source solve accounts for the forcing term, the curl correction transfers the remaining boundary mismatch, and the scalar recovery step reconstructs a field from the corrected weighted gradient. Since this mechanism is formulated before discretization, it provides a direct boundary treatment for noninterpolatory meshfree Galerkin spaces.
	
	The formulation developed here makes four contributions. First, it introduces a natural decomposition method for essential boundary conditions in noninterpolatory meshfree Galerkin spaces, with boundary data transferred before discretization and without boundary parameter tuning. Second, for topologically trivial single domains with connected boundary, it proves the continuous source, curl, and recovery reconstruction and identifies the $H(\curl)$ vector potential that carries the two tangential components of a surface mismatch. Third, the transfer formulation is extended to interface settings through broken $H(\curl)$ spaces and mean closures. Fourth, it supplies the projected error lens used in the numerical study: the final recovery error is separated into recovery space approximation and the upstream error component visible to that recovery space, with a sequential accounting estimate recording approximation, quadrature, stability, and algebraic residual terms.
	
	The rest of the paper follows this mechanism. Section~\ref{sec:3d-natural} gives the continuous transfer construction, the $H(\curl)$ correction, the single domain equivalence proof, and the broken interface formulation. Section~\ref{sec:exp} implements the method in a global MQ RBF Galerkin setting and uses projected diagnostics to examine accuracy, transfer behavior, boundary treatment comparisons, and perturbation response. Section~\ref{sec:conclusion} summarizes the boundary transfer principle. \hyperref[app:2d]{Appendix~\ref*{app:2d}} records the planar prototype, \hyperref[app:discrete_stability]{Appendix~\ref{app:discrete_stability}} gives the detailed sequential error accounting estimate, and \hyperref[app:validation_scale_cost]{Appendix~\ref{app:validation_scale_cost}} reports the computational cost of the dense global RBF realization.

	\section{Continuous Natural Decomposition in Three Dimensions}
	\label{sec:3d-natural}

	This section establishes the continuous source, curl, and recovery mechanism before any RBF space or quadrature rule is introduced. For a problem with essential boundary data, let $u$ denote the solution. The construction first separates a natural source field $\widetilde u$. The remaining information is the boundary mismatch
	\[
	\rho=g-T\widetilde u ,
	\]
	which determines the missing part of the weighted gradient rather than a new set of boundary coefficients. The central point is that, in the topologically trivial single domain setting with connected boundary, this missing component lies in the weighted curl range $\A^{-1}\curl H(\curl;\Omega)$. Thus it can be written as $\A^{-1}\curl\boldsymbol\phi$ for an $H(\curl)$ vector potential, and the scalar solution is then recovered from the corrected weighted gradient after fixing the boundary mean.
	
	This range statement is the three dimensional core of NDM. It replaces the scalar potential used in the planar natural decomposition of~\citet{YuZhang2025}, because a surface mismatch has two independent tangential components. Section~\ref{subsec:method_one_page} isolates the trace obstruction and states the transfer formulation. Section~\ref{subsec:3d-unified} formulates the three continuous subproblems. Section~\ref{subsec:transfer_mechanism} proves the weighted curl representation and continuous equivalence. Section~\ref{subsec:3d-interface} formulates the corresponding broken interface transfer setting and mean closures used in the numerical tests.

	\subsection{Boundary Admissibility Mismatch and Natural Transfer Route}
	\label{subsec:method_one_page}
	\label{subsec:admissibility_obstruction}
	
	The original formulation with essential boundary data seeks
	\[
	u\in H^1_g(\Omega):=\{w\in H^1(\Omega):Tw=g\},
	\qquad
	v\in H^1_0(\Omega)=\ker T ,
	\]
	where $T$ is the trace operator. The boundary condition is therefore a constraint on the continuous trace. It is not a prescription of algebraic coefficients or of values at a finite set of boundary points.
	
	Let $V_h$ be a noninterpolatory meshfree space generated, for example, by RBFs, MLS, or RKPM. A typical trial function has the form
	\[
	v_h(x)=\sum_{j=1}^{N} d_j\Phi_j(x),
	\]
	where the coefficients $d_j$ are not generally the nodal values $v_h(x_j)$. Even after a cardinal transformation with $L_j(x_i)=\delta_{ij}$, imposing zero values at boundary nodes gives only
	\[
	V^{\rm nodal}_{h,0}
	:=
	\{v_h\in V_h:\ v_h(x_i)=0,\ x_i\in\Gamma_h\},
	\]
	which need not coincide with $V_h\cap H^1_0(\Omega)$. In particular,
	\[
	v_h(x_i)=0,\quad x_i\in\Gamma_h
	\quad\not\Rightarrow\quad
	v_h|_\Gamma=0 .
	\]
	
	This is the point at which the trace mismatch enters the weak form. For a sufficiently smooth $v_h\in V^{\rm nodal}_{h,0}$ and exact solution $u$, integration by parts gives
	\[
	(\A\nabla u,\A\nabla v_h)_\Omega-(f,v_h)_\Omega
	=
	\langle \A^2\nabla u\cdot n,\ v_h\rangle_\Gamma .
	\]
	
	The boundary term vanishes for $v_h\in H^1_0(\Omega)$, but nodal cancellation on $\Gamma_h$ does not generally imply this trace condition. The obstruction is therefore an admissibility obstruction, not merely a pointwise boundary residual.
	
	NDM avoids this obstruction by transferring the trace data before the meshfree space is chosen. Starting from the forcing term $f$ and boundary data $g$, it computes a natural source field $\widetilde u$ and forms the residual trace data
	\[
	\rho=g-T\widetilde u .
	\]
	
	The mismatch $\rho$ is not imposed as boundary coefficients. Instead, its tangential derivative defines the natural datum for an $H(\curl)$ correction. The induced weighted correction field is
	\[
	\mathcal R_{\A}\boldsymbol\phi
	:=
	\A^{-1}(\nabla\times\boldsymbol\phi).
	\]
	
	Finally, a scalar field $u^\star$ is recovered from the corrected weighted gradient, with the mean fixed by the boundary data. Schematically,
	\[
	\boxed{
		(f,g)
		\longrightarrow
		\widetilde u
		\longrightarrow
		\mathcal R_{\A}\boldsymbol\phi
		=
		\A^{-1}(\nabla\times\boldsymbol\phi)
		\longrightarrow
		u^\star .
	}
	\]
	
	\subsection{Single Domain Formulation}
	\label{subsec:3d-unified}
	
	We state the continuous NDM for a single domain variable coefficient problem. Let $\Omega\subset\mathbb R^3$ be a bounded Lipschitz domain, let $\Gamma=\partial\Omega$, and let $\nn$ be the outward unit normal. Define
	\[
	\HoneG
	:=
	\left\{
	v\in H^1(\Omega):\int_\Gamma v\,dS=0
	\right\},
	\qquad
	H(\curl;\Omega)
	:=
	\left\{
	\boldsymbol\psi\in L^2(\Omega)^3:
	\nabla\times\boldsymbol\psi\in L^2(\Omega)^3
	\right\}.
	\]
	
	For a sufficiently smooth vector field $\boldsymbol\psi$, set
	\[
	\boldsymbol\psi_t
	:=
	\nn\times(\boldsymbol\psi\times\nn).
	\]
	
	If $\widetilde u$ is the auxiliary field obtained from the first natural subproblem, the boundary mismatch is
	\begin{equation}
		\label{eq:rho_boundary_mismatch}
		\rho:=g-T\widetilde u .
	\end{equation}
	
	For smooth data, its rotated tangential derivative is
	\begin{equation}
		\label{eq:3d-rotated-tan-data}
		\mathbf q
		:=
		\nn\times\nabla_\Gamma\rho
		=
		\nn\times\nabla_\Gamma(g-\widetilde u).
	\end{equation}
	
	This is the three dimensional counterpart of the two dimensional scalar tangential derivative $\partial_{\tg}(g-\widetilde u|_\Gamma)$. The passage from one tangential component to two tangential components is the reason that the scalar potential used in the planar construction is replaced here by an $H(\curl)$ vector potential.
	
	For general $\rho\in H^{1/2}(\Gamma)$, the surface expression $\nn\times\nabla_\Gamma\rho$ need not be a classical $L^2$ tangential field. We therefore define the corresponding tangential mismatch functional by lifting. Choose any $E\rho\in H^1(\Omega)$ satisfying $T(E\rho)=\rho$ and set
	\begin{equation}
		\label{eq:curl_rhs_lifting}
		\mathcal F_\rho(\boldsymbol\psi)
		:=
		-(\nabla E\rho,\nabla\times\boldsymbol\psi)_\Omega,
		\qquad
		\boldsymbol\psi\in H(\curl;\Omega).
	\end{equation}
	
	This definition is not a pointwise replacement of the surface gradient. It is the weak functional whose smooth counterpart is the boundary pairing in \eqref{eq:curl_rhs_surface_representation}.
	
	The following proposition records the compatibility property needed for the semidefinite curl curl correction.
	
	\begin{proposition}
		\label{prop:curl_rhs_compatibility}
		Let $\rho\in H^{1/2}(\Gamma)$ and define $\mathcal F_\rho$ by \eqref{eq:curl_rhs_lifting}. Then $\mathcal F_\rho$ is independent of the chosen lifting. Moreover, for any fixed lifting $E\rho\in H^1(\Omega)$,
		\begin{equation}
			|\mathcal F_\rho(\boldsymbol\psi)|
			\le
			C_\A\|E\rho\|_{H^1(\Omega)}
			\|\A^{-1}(\nabla\times\boldsymbol\psi)\|_{L^2(\Omega)},
			\qquad
			\forall \boldsymbol\psi\in H(\curl;\Omega),
			\label{eq:curl_rhs_quotient_bound}
		\end{equation}
		where $C_\A$ depends only on the $L^\infty$ bound of $\A$. Hence $\mathcal F_\rho$ descends to a continuous functional on the quotient space obtained by identifying vector potentials with the same induced correction field $\A^{-1}\curl\boldsymbol\psi$. In particular,
		\begin{equation}
			\mathcal F_\rho(\boldsymbol\psi)=0
			\qquad
			\forall\boldsymbol\psi\in\ker(\curl).
			\label{eq:curl_rhs_kernel_compatibility}
		\end{equation}
		
		If $\rho$, $\Gamma$, and $\boldsymbol\psi$ are sufficiently smooth, then
		\begin{equation}
			\mathcal F_\rho(\boldsymbol\psi)
			=
			\left\langle
			\nn\times\nabla_\Gamma\rho,
			\boldsymbol\psi_t
			\right\rangle_\Gamma .
			\label{eq:curl_rhs_surface_representation}
		\end{equation}
		
		Thus the right hand side of the curl correction is well defined on the quotient by curl free potentials. In the range setting used in Proposition~\ref{prop:checkable_hodge_range}, this quotient functional is represented by an actual induced correction field.
	\end{proposition}
	
	\begin{proof}
		Let $E_1\rho$ and $E_2\rho$ be two $H^1$ liftings of $\rho$. Then $z:=E_1\rho-E_2\rho\in H^1_0(\Omega)$. By density of $C_c^\infty(\Omega)$ in $H^1_0(\Omega)$ and the distributional identity $\nabla\times\nabla z=0$,
		\[
		(\nabla z,\nabla\times\boldsymbol\psi)_\Omega=0,
		\qquad
		\forall\boldsymbol\psi\in H(\curl;\Omega).
		\]
		
		Thus \eqref{eq:curl_rhs_lifting} does not depend on the lifting. The bound \eqref{eq:curl_rhs_quotient_bound} follows from Cauchy's inequality and the uniform boundedness of $\A$:
		\[
		|\mathcal F_\rho(\boldsymbol\psi)|
		\le
		\|\nabla E\rho\|_{L^2(\Omega)}
		\|\nabla\times\boldsymbol\psi\|_{L^2(\Omega)}
		\le
		C_\A\|E\rho\|_{H^1(\Omega)}
		\|\A^{-1}(\nabla\times\boldsymbol\psi)\|_{L^2(\Omega)}.
		\]
		
		If $\boldsymbol\psi\in\ker(\curl)$, then \eqref{eq:curl_rhs_kernel_compatibility} follows immediately. For smooth functions, Green's formula and $\nabla\times\nabla(E\rho)=0$ give
		\[
		(\nabla E\rho,\nabla\times\boldsymbol\psi)_\Omega
		=
		\left\langle
		\nn\times\boldsymbol\psi,
		\nabla E\rho
		\right\rangle_\Gamma .
		\]
		
		Since $\nn\times\boldsymbol\psi$ is tangential, only the tangential part of $\nabla E\rho$ contributes on $\Gamma$. Because $T(E\rho)=\rho$, this tangential part is $\nabla_\Gamma\rho$, and therefore
		\[
		(\nabla E\rho,\nabla\times\boldsymbol\psi)_\Omega
		=
		\left\langle
		\nn\times\boldsymbol\psi,
		\nabla_\Gamma\rho
		\right\rangle_\Gamma .
		\]
		
		Using
		\[
		-\left(\nn\times\boldsymbol\psi\right)\cdot\nabla_\Gamma\rho
		=
		\left(\nn\times\nabla_\Gamma\rho\right)\cdot\boldsymbol\psi_t,
		\]
		we obtain \eqref{eq:curl_rhs_surface_representation}.
	\end{proof}
	
	Consider a variable coefficient problem with essential boundary data
	\begin{equation}
		\label{eq:3d-unified-strong}
		-\operatorname{div}(\A^2\nabla u)=f \quad \text{in }\Omega,
		\qquad
		u=g \quad \text{on }\Gamma,
	\end{equation}
	where $\A(x)$ is the symmetric square root of the physical diffusion tensor, so that the physical tensor is $\A^2=\A^\top\A$. In the derivation below, $\A=\A^\top$ is assumed uniformly positive definite and uniformly bounded. The Poisson case corresponds to $\A=I$. The reconstruction consists of three natural subproblems.
	
	\medskip
	\noindent\textbf{Step 1: source to auxiliary scalar field.}
	Find $\widetilde u\in\HoneG$ such that
	\begin{equation}
		\label{eq:3d-unified-s1}
		(\A\nabla\widetilde u,\A\nabla v)
		=
		\ip{f}{v}_{(\HoneG)'\times\HoneG},
		\qquad
		\forall v\in\HoneG .
	\end{equation}
	
	This Neumann type problem carries the forcing term, and the zero boundary mean fixes the additive constant.
	
	\medskip
	\noindent\textbf{Step 2: boundary mismatch to curl correction.}
	Find a vector potential $\boldsymbol\phi\in H(\curl;\Omega)$ such that
	\begin{equation}
		\label{eq:3d-unified-s2}
		(\A^{-1}(\nabla\times\boldsymbol\phi),
		\A^{-1}(\nabla\times\boldsymbol\psi))_\Omega
		=
		\mathcal F_{g-T\widetilde u}(\boldsymbol\psi),
		\qquad
		\forall \boldsymbol\psi\in H(\curl;\Omega).
	\end{equation}
	
	For smooth data, \eqref{eq:3d-unified-s2} is equivalently written as
	\begin{equation}
		\label{eq:3d-unified-s2-surface}
		(\A^{-1}(\nabla\times\boldsymbol\phi),
		\A^{-1}(\nabla\times\boldsymbol\psi))_\Omega
		=
		\ip{\nn\times\nabla_\Gamma(g-\widetilde u)}{\boldsymbol\psi_t}_{\Gamma},
		\qquad
		\forall \boldsymbol\psi\in H(\curl;\Omega).
	\end{equation}
	
	The induced correction field is
	\[
	\mathcal R_\A\boldsymbol\phi
	:=
	\A^{-1}(\nabla\times\boldsymbol\phi).
	\]
	
	\medskip
	\noindent\textbf{Step 3: projected weighted gradient recovery.}
	Find $u_c\in H^1(\Omega)$ such that
	\begin{equation}
		\label{eq:3d-unified-s3}
		(\A\nabla u_c,\A\nabla v)_\Omega
		=
		(\A\nabla\widetilde u-\mathcal R_\A\boldsymbol\phi,\A\nabla v)_\Omega,
		\qquad
		\forall v\in H^1(\Omega).
	\end{equation}
	
	The additive constant is fixed by the boundary mean:
	\begin{equation}
		\label{eq:3d-unified-recover}
		u^\star=u_c-C,
		\qquad
		C
		=
		\frac{1}{|\Gamma|}
		\int_\Gamma (u_c-g)\,dS .
	\end{equation}
	
	Thus the recovered weighted gradient is built from
	\[
	\A\nabla\widetilde u-\A^{-1}(\nabla\times\boldsymbol\phi).
	\]
	
	This is the three dimensional weighted gradient counterpart of the two dimensional field $\nabla\widetilde u-\curl\varphi$.
	
	\subsection{Curl Range and Equivalence}
	\label{subsec:transfer_mechanism}
	\label{subsec:on_assumption_H}
	
	The three step formulation above is equivalent to the original problem with essential boundary data when the missing weighted correction field belongs to the curl range. This subsection proves that this is the case in the standard topologically trivial single domain setting. The key observation is that, after the first natural solve, the difference between the auxiliary weighted gradient and the exact weighted gradient is orthogonal to all homogeneous weighted gradients.
	
	Define
	\[
	\mathcal X_\A
	:=
	\left\{
	\mathbf p\in L^2(\Omega)^3:
	(\mathbf p,\A\nabla v)_\Omega=0,
	\forall v\in H^1_0(\Omega)
	\right\}.
	\]
	
	If $u\in\Honeg$ is the solution of the original problem with essential boundary data and $\widetilde u\in H^1_\Gamma(\Omega)$ is the solution of the first natural subproblem, then subtracting the two weak forms gives
	\[
	\A\nabla\widetilde u-\A\nabla u\in \mathcal X_\A .
	\]
	
	The correction step represents this weighted orthogonal complement component by an $H(\curl)$ vector potential. The following assumption records the topological setting in which the representation follows from the standard curl range theorem.
	
	\begin{assumption}
		\label{ass:weighted_helmholtz_closure}
		Assume that the following conditions hold.
		\begin{enumerate}
			\renewcommand{\labelenumi}{(H\arabic{enumi})}
			\item The domain $\Omega\subset\mathbb R^3$ is bounded, Lipschitz, and has connected boundary. Moreover, it is topologically trivial in the sense that the standard curl range relation
			\begin{equation}
				\label{eq:curl_range_condition}
				\nabla\times H(\curl;\Omega)
				=
				\left\{
				\mathbf r\in H(\operatorname{div};\Omega):
				\operatorname{div}\mathbf r=0,
				\ \langle \mathbf r\cdot\nn,1\rangle_{\Gamma}=0
				\right\}
			\end{equation}
			holds. A topologically trivial Lipschitz polyhedron with connected boundary is a standard example of this setting; see, for example, \citet{GiraultRaviart1986} and \citet{ArnoldFalkWinther2010} for background on Hodge decompositions, de Rham complexes, and finite element exterior calculus.
			
			\item The coefficient $\A=\A^\top\in L^\infty(\Omega)^{3\times3}$ is uniformly bounded and uniformly positive definite.
			
			\item The data are regular enough for the problem with essential boundary conditions, the first natural subproblem, and the lifting functional \eqref{eq:curl_rhs_lifting} to be well defined. In particular, the scalar mismatch $\rho=g-T\widetilde u$ belongs to $H^{1/2}(\Gamma)$. For smooth or piecewise smooth boundaries, the equivalent surface pairings are interpreted patchwise.
		\end{enumerate}
	\end{assumption}
	
	\begin{proposition}
		\label{prop:checkable_hodge_range}
		Let Assumption~\ref{ass:weighted_helmholtz_closure} hold. Let $u\in H^1_g(\Omega)$ solve the problem with essential boundary conditions and let $\widetilde u\in H^1_\Gamma(\Omega)$ solve the first natural subproblem. Define
		\[
		\mathbf p:=\A\nabla\widetilde u-\A\nabla u .
		\]
		
		Then there exists $\boldsymbol\phi\in H(\curl;\Omega)$ such that
		\[
		\mathbf p=\A^{-1}(\nabla\times\boldsymbol\phi).
		\]
	\end{proposition}
	
	\begin{proof}
		Subtracting the weak form of the original problem with essential boundary data from Subproblem~\eqref{eq:3d-unified-s1} gives
		\[
		(\A\nabla\widetilde u-\A\nabla u,\A\nabla v)_\Omega=0,
		\qquad
		\forall v\in H^1_0(\Omega),
		\]
		and hence $\mathbf p\in\mathcal X_\A$. Set $\mathbf r:=\A\mathbf p$. Since $\A=\A^\top$,
		\[
		(\mathbf r,\nabla v)_\Omega
		=
		(\A\mathbf p,\nabla v)_\Omega
		=
		(\mathbf p,\A\nabla v)_\Omega
		=0,
		\qquad
		\forall v\in H^1_0(\Omega).
		\]
		
		Thus $\operatorname{div}\mathbf r=0$ in the distributional sense. Since $\mathbf r\in L^2(\Omega)^3$ and $\operatorname{div}\mathbf r=0\in L^2(\Omega)$, we have $\mathbf r\in H(\operatorname{div};\Omega)$. The connected boundary condition gives the required flux compatibility. Indeed, by the normal trace formula in $H(\operatorname{div};\Omega)$,
		\[
		\langle \mathbf r\cdot\nn,1\rangle_\Gamma
		=
		(\operatorname{div}\mathbf r,1)_\Omega+(\mathbf r,\nabla 1)_\Omega
		=0.
		\]
		
		The curl range relation \eqref{eq:curl_range_condition} therefore gives a vector potential $\boldsymbol\phi\in H(\curl;\Omega)$ satisfying $\nabla\times\boldsymbol\phi=\mathbf r$. Consequently,
		\[
		\mathbf p=\A^{-1}\mathbf r=\A^{-1}(\nabla\times\boldsymbol\phi).
		\]
	\end{proof}

	Proposition~\ref{prop:checkable_hodge_range} also closes the solvability point for Subproblem~\eqref{eq:3d-unified-s2} in the single domain theory. Let $\rho=g-T\widetilde u$. Since $u|_\Gamma=g$, the function $u-\widetilde u$ is an admissible lifting of $\rho$. The vector potential $\boldsymbol\phi_0$ constructed from
	\[
	\A\nabla\widetilde u-\A\nabla u
	=
	\A^{-1}(\nabla\times\boldsymbol\phi_0)
	\]
	represents the quotient functional $\mathcal F_\rho$. Indeed, for every $\boldsymbol\psi\in H(\curl;\Omega)$,
	\[
	\bigl(\A^{-1}(\nabla\times\boldsymbol\phi_0),
	\A^{-1}(\nabla\times\boldsymbol\psi)\bigr)_\Omega
	=
	\mathcal F_\rho(\boldsymbol\psi).
	\]
	
	Thus the second natural subproblem has a solution in the stated range setting. If two vector potentials solve it, their induced correction fields coincide in $L^2(\Omega)^3$. The potential is therefore unique only modulo the curl free kernel, while the field $\mathcal R_\A\boldsymbol\phi$ used in the recovery step is unique.
	
	\begin{proposition}
		\label{prop:3d_equivalence}
		Let Assumption~\ref{ass:weighted_helmholtz_closure} hold. Let $u\in\Honeg$ solve the original problem with essential boundary data
		\[
		-\nabla\cdot(\A^2\nabla u)=f\quad\text{in }\Omega,
		\qquad
		u=g\quad\text{on }\Gamma .
		\]
		
		Let $\widetilde u$, $\boldsymbol\phi$, and $u_c$ satisfy the first, second, and third subproblems in the three dimensional natural decomposition, respectively, and define the final solution $u^\star$ by the boundary mean closure condition. Then
		\[
		u^\star=u .
		\]
	\end{proposition}
	
	\begin{proof}
		By Proposition~\ref{prop:checkable_hodge_range}, for
		\[
		\mathbf p:=\A\nabla\widetilde u-\A\nabla u
		\]
		there exists $\boldsymbol\phi_0\in H(\curl;\Omega)$ such that
		\[
		\mathbf p
		=
		\mathcal R_\A\boldsymbol\phi_0
		=
		\A^{-1}(\nabla\times\boldsymbol\phi_0).
		\]
		
		Let $\rho=g-T\widetilde u$. Since $u|_\Gamma=g$, $u-\widetilde u$ is an admissible lifting of $\rho$. For any $\boldsymbol\psi\in H(\curl;\Omega)$, using $\A=\A^\top$ gives
		\[
		\begin{aligned}
			\bigl(\A^{-1}(\nabla\times\boldsymbol\phi_0),
			\A^{-1}(\nabla\times\boldsymbol\psi)\bigr)_\Omega
			&=
			\bigl(\A\nabla\widetilde u-\A\nabla u,
			\A^{-1}(\nabla\times\boldsymbol\psi)\bigr)_\Omega \\
			&=
			(\nabla\widetilde u-\nabla u,\nabla\times\boldsymbol\psi)_\Omega \\
			&=
			-(\nabla(u-\widetilde u),\nabla\times\boldsymbol\psi)_\Omega \\
			&=
			\mathcal F_\rho(\boldsymbol\psi).
		\end{aligned}
		\]
		
		Thus $\boldsymbol\phi_0$ is an admissible representative for the second natural subproblem. In the smooth case this identity is exactly the surface formula \eqref{eq:3d-unified-s2-surface}. If another vector potential solves the same semidefinite curl curl problem, the induced field $\mathcal R_\A\boldsymbol\phi$ is the same: testing the difference of two solutions by itself gives zero $L^2$ norm of the difference of the induced correction fields. Hence the third step depends only on
		\[
		\mathcal R_\A\boldsymbol\phi
		=
		\A\nabla\widetilde u-\A\nabla u .
		\]
		
		Substitution into the third natural subproblem gives
		\[
		(\A\nabla u_c,\A\nabla v)_\Omega
		=
		(\A\nabla u,\A\nabla v)_\Omega,
		\qquad
		\forall v\in H^1(\Omega).
		\]
		
		Therefore $u_c-u$ is a constant. The boundary mean closure condition fixes this constant, and therefore $u^\star=u$.
	\end{proof}
	
	The proof also identifies the only topological input used in the single domain equivalence: the curl range relation in Assumption~\ref{ass:weighted_helmholtz_closure}.
	
	\begin{remark}
		\label{rem:assumption_H_scope}
		The equivalence result above is a direct single domain statement under Assumption~\ref{ass:weighted_helmholtz_closure}. If the domain has several boundary components or nontrivial cohomology, the curl range may contain additional harmonic compatibility components. If the problem is formulated with internal interfaces, nonzero jumps, or broken regularity, the correction should instead be written in the corresponding broken $H(\curl)$ setting or with an appropriate range projection. These are compatibility ingredients of the functional setting, not penalty parameters, Nitsche stabilization parameters, or multiplier space choices. The discontinuous coefficient, interface, and singular geometry tests in Section~\ref{sec:exp} extend the discrete transfer study beyond the direct single-domain equivalence theorem. For the interface setting, Proposition~\ref{prop:conditional_interface_equivalence} below records the same equivalence mechanism under additional broken range and lifting assumptions.
	\end{remark}
	
	\subsection{Interface Transfer Formulation}
	\label{subsec:3d-interface}
	
	The interface formulation uses the same natural transfer mechanism. The normal flux jump is assigned to the first natural problem, while the solution jump enters the second problem through its tangential derivative on the interface. Thus the interface case applies the same decomposition to exterior boundary and interior interface data in parallel. After the mean closures are specified, Proposition~\ref{prop:conditional_interface_equivalence} states the corresponding interface reconstruction result under the broken range and lifting assumptions needed for this piecewise setting.
	
	Let $\Omega_1$ and $\Omega_2$ be two disjoint open subdomains satisfying
	\[
	\Omega=\Omega_1\cup\Gamma_0\cup\Omega_2,
	\qquad
	\Omega_1\cap\Omega_2=\emptyset,
	\]
	with
	\[
	\Gamma_0=\partial\Omega_1\cap\partial\Omega_2\cap\Omega,
	\qquad
	\Gamma=\partial\Omega .
	\]
	Thus $\Omega_1\cup\Omega_2$ denotes the broken interior on which the differential equation is imposed, while the jump conditions are imposed separately on $\Gamma_0$.
	
	Let $\nn_0$ be the interface normal directed from $\Omega_1$ to $\Omega_2$. We use the jump convention
	\[
	[q]_{\Gamma_0}=q_1|_{\Gamma_0}-q_2|_{\Gamma_0},
	\]
	where the subscripts 1 and 2 denote traces taken from the $\Omega_1$ and $\Omega_2$ sides, respectively. Consider
	\begin{equation}
		\label{eq:3d-interface-strong}
		\begin{cases}
			-\operatorname{div}(\A^2\nabla u)=f & \text{in }\Omega_1\cup\Omega_2, \\
			u=g & \text{on }\Gamma, \\
			[(\A^2\nabla u)\cdot\nn_0]_{\Gamma_0}=\kappa_2, \\
			[u]_{\Gamma_0}=\kappa_1 .
		\end{cases}
	\end{equation}
	
	For discontinuous coefficients or nonzero solution jumps, the curl correction can be understood in the broken space
	\[
	H(\curl;\Omega_1\cup\Omega_2)
	:=
	\{\boldsymbol\psi:
	\boldsymbol\psi_i\in H(\curl;\Omega_i),\ i=1,2\}.
	\]
	
	A globally conforming $H(\curl;\Omega)$ potential requires the corresponding tangential trace compatibility across $\Gamma_0$. The numerical interface tests below use the piecewise interpretation, which is the natural setting for piecewise recovery.
	
	In the first subproblem, the flux jump enters as a natural term:
	\begin{equation}
		\label{eq:3d-interface-s1}
		(\A\nabla\widetilde u,\A\nabla v)_\Omega
		=
		\ip{f}{v-\bar v_\Gamma}
		+
		\ip{\kappa_2}{v-\bar v_\Gamma}_{\Gamma_0},
		\qquad
		\forall v\in H^1(\Omega),
	\end{equation}
	where
	\[
	\bar v_\Gamma
	:=
	\frac{1}{|\Gamma|}
	\int_\Gamma v\,dS .
	\]
	
	In the second subproblem, the exterior boundary mismatch and the interface solution jump jointly generate tangential correction data:
	\[
	\rho_\Gamma=g-T\widetilde u,
	\qquad
	\rho_0=\kappa_1 .
	\]
	
	For smooth data,
	\[
	\mathbf q_\Gamma
	=
	\nn\times\nabla_\Gamma\rho_\Gamma,
	\qquad
	\mathbf q_0
	=
	\nn_0\times\nabla_{\Gamma_0}\rho_0
	=
	\nn_0\times\nabla_{\Gamma_0}\kappa_1 .
	\]
	
	Accordingly, the vector potential satisfies the broken lifting form; for smooth data it becomes
	\begin{equation}
		\label{eq:3d-interface-s2}
		(\A^{-1}(\nabla\times\boldsymbol\phi),
		\A^{-1}(\nabla\times\boldsymbol\psi))_{\Omega_1\cup\Omega_2}
		=
		\ip{\mathbf q_\Gamma}{\boldsymbol\psi_t}_{\Gamma}
		+
		\ip{\mathbf q_0}{\boldsymbol\psi_{t,0}}_{\Gamma_0},
		\qquad
		\forall\boldsymbol\psi\in H(\curl;\Omega_1\cup\Omega_2),
	\end{equation}
	where
	\[
	\boldsymbol\psi_{t,0}
	:=
	\nn_0\times(\boldsymbol\psi\times\nn_0).
	\]
	
	The third subproblem recovers the solution piecewise on the two subdomains:
	\begin{equation}
		\label{eq:3d-interface-s3}
		(\A\nabla u_i^c,\A\nabla v)_{\Omega_i}
		=
		(\A\nabla\widetilde u-\mathcal R_\A\boldsymbol\phi,\A\nabla v)_{\Omega_i},
		\qquad
		\forall v\in H^1(\Omega_i),
		\quad i=1,2 .
	\end{equation}
	
	Each recovered piece is determined only up to an additive constant. This point must be closed explicitly in the interface formulation. Let
	\[
	\Gamma_i:=\partial\Omega_i\cap\Gamma,
	\qquad i=1,2,
	\]
	and set
	\begin{equation}
		\label{eq:3d-interface-piecewise-final}
		u_i^\star=u_i^c-C_i,
		\qquad i=1,2 .
	\end{equation}
	
	The constants $C_1$ and $C_2$ are fixed by one exterior boundary mean condition and one interface jump mean condition:
	\begin{equation}
		\label{eq:3d-interface-closure-weak}
		\sum_{i=1}^2
		\int_{\Gamma_i}(u_i^c-C_i-g)\,dS=0,
		\qquad
		\int_{\Gamma_0}
		\bigl(u_1^c-C_1-u_2^c+C_2-\kappa_1\bigr)\,dS=0 .
	\end{equation}
	
	Equivalently,
	\begin{equation}
		\label{eq:3d-interface-closure-system}
		|\Gamma_1|C_1+|\Gamma_2|C_2
		=
		\sum_{i=1}^2\int_{\Gamma_i}(u_i^c-g)\,dS,
		\qquad
		C_1-C_2
		=
		\frac{1}{|\Gamma_0|}
		\int_{\Gamma_0}(u_1^c-u_2^c-\kappa_1)\,dS .
	\end{equation}
	
	\begingroup
	
	\begin{proposition}
		\label{prop:conditional_interface_equivalence}
		Assume that the interface problem \eqref{eq:3d-interface-strong} has a sufficiently regular solution $u_i$ on each $\Omega_i$, and let $\widetilde u$ solve \eqref{eq:3d-interface-s1}. Suppose, in addition, that the broken lifting in \eqref{eq:3d-interface-s2} represents the exterior trace mismatch $g-T\widetilde u$ and the interface jump $\kappa_1$, and that the corresponding subdomain mismatch fields
		\[
		\mathbf p_i
		:=
		\A_i\nabla\widetilde u_i-\A_i\nabla u_i,
		\qquad
		\A_i:=\A|_{\Omega_i},
		\]
		belong to the broken weighted curl ranges, namely
		\[
		\mathbf p_i
		=
		\A_i^{-1}(\nabla\times\boldsymbol\phi_i),
		\qquad
		\boldsymbol\phi_i\in H(\curl;\Omega_i),
		\qquad i=1,2.
		\]
		Then the piecewise recovery \eqref{eq:3d-interface-s3}, followed by the two mean closures \eqref{eq:3d-interface-closure-system}, reconstructs the interface solution: $u_i^\star=u_i$ in $\Omega_i$, $i=1,2$.
	\end{proposition}
	
	\begin{proof}
		Under the assumed broken lifting identity, the potential $\boldsymbol\phi=(\boldsymbol\phi_1,\boldsymbol\phi_2)$ is an admissible representative of the interface curl correction in \eqref{eq:3d-interface-s2}. Hence on each subdomain
		\[
		\A\nabla\widetilde u-\mathcal R_\A\boldsymbol\phi
		=
		\A_i\nabla\widetilde u_i-\A_i^{-1}(\nabla\times\boldsymbol\phi_i)
		=
		\A_i\nabla u_i .
		\]
		
		The recovery equation \eqref{eq:3d-interface-s3} therefore gives
		\[
		(\A_i\nabla(u_i^c-u_i),\A_i\nabla v)_{\Omega_i}=0
		\qquad
		\forall v\in H^1(\Omega_i),
		\]
		so $u_i^c-u_i=d_i$ is constant on each connected subdomain. Since $u=g$ on $\Gamma$ and $[u]_{\Gamma_0}=\kappa_1$, the two closure equations reduce to
		\[
		|\Gamma_1|(d_1-C_1)+|\Gamma_2|(d_2-C_2)=0,
		\qquad
		(d_1-C_1)-(d_2-C_2)=0 .
		\]
		
		These equations imply $C_i=d_i$ whenever $|\Gamma_1|+|\Gamma_2|>0$, and thus $u_i^\star=u_i^c-C_i=u_i$ on each subdomain. This proves the claimed piecewise reconstruction result under the stated range and lifting assumptions.
	\end{proof}
	\endgroup
	
	Thus the interface extension keeps the same logic as the single domain formulation: flux type data are placed in the auxiliary scalar solve, tangential boundary mismatch is carried by the curl correction, and the final solution is obtained by piecewise weighted gradient recovery plus the mean closures \eqref{eq:3d-interface-closure-system}.

	\section{Discrete Natural Transfer and Numerical Experiments}
	\label{sec:exp}
	
	This section tests the discrete NDM in a global MQ RBF Galerkin realization. The RBF space serves as a transparent noninterpolatory test bed; the reported quantities include the final weighted gradient error, the projected split \eqref{eq:numerical_error_split}, subproblem conditioning, and the response to prescribed boundary perturbations. Section~\ref{subsec:exp_setup} fixes the discrete transfer setting and diagnostics. Section~\ref{subsec:ndm_accuracy_transfer} tests the planar prototype, the three dimensional vector potential, and interface transfer. Section~\ref{subsec:ebc_mechanism_tests} compares boundary treatment mechanisms in the same noninterpolatory space. Section~\ref{subsec:perturbation_transfer} examines boundary information propagation through the projected recovery step.
	
	The experiments follow the equation level chain \eqref{eq:3d-unified-s1} through \eqref{eq:3d-unified-recover}: source solve, curl transfer, weighted gradient projection, and mean closure. Figure~\ref{fig:ndm_flow} gives the compact visual form of this chain.
	
	\begin{figure}[htbp]
		\centering
		\includegraphics[width=0.86\linewidth]{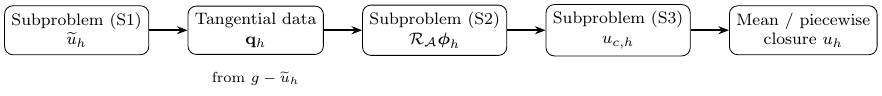}
		\caption{Data flow of the discrete NDM corresponding to the single domain reconstruction \eqref{eq:3d-unified-s1} through \eqref{eq:3d-unified-recover} and the interface reconstruction \eqref{eq:3d-interface-s1} through \eqref{eq:3d-interface-closure-system}. Boundary information enters through tangential natural data and the curl correction field, not through a penalty term, Nitsche stabilization, or multiplier constraint.}
		\label{fig:ndm_flow}
	\end{figure}

	\subsection{Discrete Transfer Setting and Diagnostics}
	\label{subsec:exp_setup}
	
	Let $V_h\subset H^1(\Omega)$ be a scalar approximation space and let $\mathbf V_h:=V_h^3$ be the associated vector valued space. The continuous decomposition does not prescribe a particular meshfree approximation; for the numerical tests we use a global multiquadric (MQ) RBF Galerkin space as a representative noninterpolatory discretization~\citep{Babuska2003,Wendland2004}. The weighted gradient recovery space $G_{\A,h}$ is the image of $V_h$ under the discrete weighted gradient operator from the third subproblem. The symbol $\Pi_{\A,h}$ denotes the $L^2(\Omega)^3$-orthogonal projection onto $G_{\A,h}$.
	
	The discrete NDM is then assembled as three sequential solves:
	\begin{enumerate}
		\renewcommand{\labelenumi}{(S\arabic{enumi})}
		\item Find $\widetilde u_h\in V_h\cap\HoneG$ satisfying \eqref{eq:3d-unified-s1} for all $v_h\in V_h\cap\HoneG$.
		\item Find $\boldsymbol\phi_h\in\mathbf V_h$ satisfying \eqref{eq:3d-unified-s2} for all $\boldsymbol\psi_h\in\mathbf V_h$. The curl curl matrix is symmetric positive semidefinite. Its null space consists of discrete curl free modes, including discrete gradients when the chosen vector space contains the corresponding scalar gradients, and may also include numerical null modes caused by the RBF basis and quadrature. Thus $\boldsymbol\phi_h$ is not unique, but null space components do not contribute to the induced correction field $\mathcal R_\A\boldsymbol\phi_h=\A^{-1}(\nabla\times\boldsymbol\phi_h)$. We solve this semidefinite system with MINRES QLP~\citep{ChoiPaigeSaunders2011}, avoiding an explicit gauge condition.
		\item Find $u_{c,h}\in V_h$ satisfying \eqref{eq:3d-unified-s3} for all $v_h\in V_h$, and apply the boundary mean closure \eqref{eq:3d-unified-recover}.
	\end{enumerate}
	
	Thus Subproblems~(S1), (S2), and~(S3) implement the continuous transfer mechanism at the discrete level, with the last solve projecting the corrected weighted field into the scalar recovery space before the boundary mean closure is applied.
	
	Unless stated otherwise, the approximation uses the global MQ kernel $\psi(r)=\sqrt{r^2+c^2}$. Discontinuous and interface problems use piecewise global spaces. The discrete degrees of freedom are associated with scattered centers or point cloud centers; $N_{\rm side}^2$ and $N_{\rm side}^3$ denote the node counts in two and three dimensions. The ratio $c/h$ is the shape parameter divided by the average point cloud spacing.
	
	The three dimensional experiments are posed on $\Omega=[-1,1]^3$ or on the corresponding interface partitions. The shape parameter ratios in the three subproblems are generally
	\[
	c_1/h=3.0,
	\qquad c_2/h=2.5,
	\qquad c_3/h=2.5 .
	\]
	
	Subproblem 2 uses a vector valued MQ potential space with quadratic polynomial augmentation, and Subproblem 3 uses a scalar MQ recovery space with cubic polynomial gradient augmentation. The semidefinite curl curl system in Subproblem 2 is solved by MINRES QLP with tolerance $10^{-7}$ or $10^{-8}$.
	
	\medskip\noindent
	\textbf{Projected transfer viewpoint.}
	The continuous equivalence proof shows that, before discretization, the missing weighted gradient component $\A\nabla\widetilde u-\A\nabla u$ is exactly represented by the induced curl correction and removed in the final scalar recovery. After discretization, however, the field entering the last recovery step is only the approximate corrected field produced by Subproblems~(S1) and~(S2). Hence the relevant computable quantity is the part of the resulting upstream error that is visible to the finite dimensional weighted gradient recovery space. The projected split below makes this transfer statement computable.
	
	\noindent
	\textbf{Error decomposition diagnostics.}
	Let $u$ denote the exact solution and let
	\[
	w_h:=\A\nabla\widetilde u_h-\mathcal R_\A\boldsymbol\phi_h
	\]
	be the intermediate weighted recovery field produced by Subproblems~(S1) and~(S2). Recall that $G_{\A,h}$ is the weighted gradient recovery space and that $\Pi_{\A,h}$ denotes the $L^2(\Omega)^3$ orthogonal projection onto $G_{\A,h}$. If Subproblem~(S3) is solved exactly with the same inner product used to define this projection, then its Galerkin equation gives the projection relation
	\[
	\A\nabla u_{c,h}=\Pi_{\A,h}w_h .
	\]
	
	When quadrature and algebraic residuals are included, this identity is perturbed by the terms recorded in Appendix~\ref{app:discrete_stability}. The diagnostic below should therefore be read as the projected part of the final recovery error, with those perturbations excluded.
	
	The projected weighted gradient error is decomposed as
	\begin{equation}
		\A\nabla u-\Pi_{\A,h} w_h
		=
		\bigl(I-\Pi_{\A,h}\bigr)(\A\nabla u)
		+
		\Pi_{\A,h}\bigl(\A\nabla u-w_h\bigr).
		\label{eq:numerical_error_split}
	\end{equation}
	
	We use the two diagnostics
	\begin{equation}
		\Eapp
		:=
		\left\|
		\bigl(I-\Pi_{\A,h}\bigr)(\A\nabla u)
		\right\|_{L^2(\Omega)},
		\qquad
		\mathcal E_{\rm proj}^{\rm up}
		:=
		\left\|
		\Pi_{\A,h}\bigl(\A\nabla u-w_h\bigr)
		\right\|_{L^2(\Omega)} .
		\label{eq:formal_diagnostics_def}
	\end{equation}
	
	The first quantity is the approximation defect of the final weighted gradient recovery space. The second quantity is the upstream error visible to that recovery space. It contains the part of the combined Subproblem~(S1) and Subproblem~(S2) error that survives the projection into $G_{\A,h}$. Thus
	\begin{equation}
		\label{eq:total_bound_discrete}
		\|\A\nabla u-\Pi_{\A,h} w_h\|_{L^2(\Omega)}
		\le
		\Eapp+\mathcal E_{\rm proj}^{\rm up} .
	\end{equation}
	
	This is the computable counterpart of the projected transfer viewpoint stated above. Appendix~\ref{app:discrete_stability} gives the corresponding sequential error accounting estimate, where the scalar step errors, curl correction defect, recovery defect, quadrature errors, boundary data approximation errors, and algebraic residuals are kept separately.
	
	\begin{proposition}
		\label{prop:conditional_discrete_consistency_main}
		Consider the topologically trivial single domain setting of Assumption~\ref{ass:weighted_helmholtz_closure}. Let $u_h$ be the final discrete reconstruction produced by Subproblems~(S1), (S2), and~(S3), and let the approximation and perturbation quantities $I_{1,h}$, $I_{2,h}$, $I_{3,h}$, $\Delta_{1,h}$, $\Delta^R_{2,h}$, $\Delta_{3,h}$, and $\varepsilon_{\rm alg,h}$ be those specified in Appendix~\ref{app:discrete_stability}. If the scalar source and recovery steps are stable on their constrained spaces, and if the induced curl correction error is controlled through the terms $I_{2,h}$ and $\Delta^R_{2,h}$ so that
		\begin{equation}
			I_{3,h}+I_{2,h}+C_{1,h}I_{1,h}
			+
			\Delta_{1,h}+\Delta^R_{2,h}+\Delta_{3,h}+\varepsilon_{\rm alg,h}
			\longrightarrow 0,
			\label{eq:main_conditional_consistency_assumption}
		\end{equation}
		then
		\begin{equation}
			\|\A\nabla(u-u_h)\|_{L^2(\Omega)}\longrightarrow 0 .
		\end{equation}
		
		This statement isolates the discrete consistency of the NDM transfer mechanism from MQ RBF specific rate estimates, which depend on approximation, quadrature, conditioning, and algebraic solver bounds.
	\end{proposition}
	
	The diagnostic $\Eapp$ also exposes polynomial alignment in the recovery space. The recovery space used below contains polynomial gradient augmentation; polynomial manufactured solutions can therefore make $\Eapp$ much smaller than it would be for a generic solution. The numerical tests therefore include a nonpolynomial three-dimensional Poisson case and an additional L-shaped singular benchmark, so that the recovery approximation component remains visible.
	
	Two additional fixed space transfer diagnostics are used in Table~\ref{tab:3d_transfer_compact_revised}. They measure the component of perturbations before recovery observed in the recovered weighted gradient. Let $\delta_{123}$ denote the difference between two fixed space runs in which the first natural subproblem is perturbed while the downstream spaces, quadratures, and evaluation rules are kept fixed. Let $\delta_{23}$ denote the difference between two fixed space runs in which the correction field entering the final recovery is perturbed while the recovery space and evaluation rule are kept fixed. When the denominators are nonzero, we define
	\[
	T_{123}
	:=
	\frac{
		\left\|
		\delta_{123}\bigl(\A\nabla u_{c,h}\bigr)
		\right\|_{L^2(\Omega)}
	}{
		\left\|
		\delta_{123}\bigl(\A\nabla\widetilde u_h\bigr)
		\right\|_{L^2(\Omega)}
	},
	\qquad
	T_{23}
	:=
	\frac{
		\left\|
		\delta_{23}\bigl(\A\nabla u_{c,h}\bigr)
		\right\|_{L^2(\Omega)}
	}{
		\left\|
		\delta_{23}\bigl(\mathcal R_\A\boldsymbol\phi_h\bigr)
		\right\|_{L^2(\Omega)}
	}.
	\]
	
	The quantity $T_{123}$ records the response of the final weighted gradient to a perturbation generated in Subproblem~(S1) and then propagated through Subproblems~(S2) and~(S3). The quantity $T_{23}$ records the response of the final weighted gradient to a perturbation introduced directly in the curl correction field before Subproblem~(S3). These values are fixed space sensitivity ratios that quantify the particular perturbation pathways used in the diagnostic runs.

	\subsection{NDM Accuracy and Transfer Diagnostics}
	\label{subsec:ndm_accuracy_transfer}
	
	The first group of tests isolates the discrete natural decomposition. The two dimensional examples provide a compact implementation check for the planar prototype, while the main evidence below concerns the genuinely three dimensional $H(\curl)$ correction and the projected recovery split.
	
	\subsubsection{Planar Consistency Check}
	\label{subsubsec:2d_consistency}
	
	As a compact check of the planar prototype recalled in \hyperref[app:2d]{Appendix~\ref*{app:2d}}, five two dimensional manufactured examples are solved on $\Omega=[-1,1]^2$:
	\begin{equation}
		-\nabla\cdot\bigl(\A^2\nabla u\bigr)=f
		\quad \text{in }\Omega,
		\qquad
		u=g\quad \text{on }\Gamma .
		\label{eq:2d_model_problem}
	\end{equation}
	
	These tests verify that the planar implementation reproduces the expected behavior before the three-dimensional vector potential correction is examined.
	
	Table~\ref{tab:2d_summary_compact_revised} combines the manufactured settings and the observed error ranges. The errors decrease in all cases. The smooth Poisson and smooth coefficient tests give the larger fitted orders, while nonsmooth coefficients, coefficient jumps, and interface jumps reduce the rates. These trends confirm the planar implementation check and prepare the three dimensional vector potential tests in Section~\ref{subsubsec:3d_vector_transfer}.
	
	\begin{table}[!htbp]
		\centering
		\caption{Compact two dimensional consistency check for the manufactured problem with essential boundary data \eqref{eq:2d_model_problem}.}
		\label{tab:2d_summary_compact_revised}
		\small
		\begin{widetable}
			\begin{tabular}{c p{0.22\textwidth} p{0.25\textwidth} c c}
				\toprule
				Case & Manufactured setting & Purpose of the check & $L^2$ error range, order & $H^1$ error range, order \\
				\midrule
				1
				& $\A=I_2$, $u=x_1^2+x_2^2+\sin(x_1+x_2)$
				& Planar Poisson case
				& $3.790{\times}10^{-3}\to7.438{\times}10^{-4}$, $1.43$
				& $8.539{\times}10^{-2}\to2.847{\times}10^{-2}$, $0.97$ \\
				\addlinespace[2pt]
				2
				& $\A=\operatorname{diag}(1+x_1^2,1)$, $u=e^{\cos(x_1+x_2^2)}$
				& Smooth coefficient case
				& $5.827{\times}10^{-3}\to1.341{\times}10^{-3}$, $1.27$
				& $1.195{\times}10^{-1}\to5.015{\times}10^{-2}$, $0.74$ \\
				\addlinespace[2pt]
				3
				& $\A=\operatorname{diag}(1+x_1^2,1+|x_2|)$, $u=e^{\cos(x_1+|x_2|^3)}$
				& Nonsmooth coefficient check
				& $7.656{\times}10^{-3}\to1.497{\times}10^{-3}$, $1.42$
				& $1.967{\times}10^{-1}\to7.418{\times}10^{-2}$, $0.85$ \\
				\addlinespace[2pt]
				4
				& $\A=\operatorname{diag}(1,\frac43-\frac23\operatorname{sgn}(x_2))$
				& Coefficient jump check
				& $3.709{\times}10^{-2}\to1.439{\times}10^{-2}$, $0.93$
				& $6.547{\times}10^{-1}\to4.223{\times}10^{-1}$, $0.34$ \\
				\addlinespace[2pt]
				5
				& Interface case with $\A|_{\Omega_1}=10I_2$ and variable $\A|_{\Omega_2}$
				& Interface jump check
				& $9.302{\times}10^{-2}\to5.052{\times}10^{-2}$, $0.76$
				& $6.119{\times}10^{-1}\to4.261{\times}10^{-1}$, $0.44$ \\
				\bottomrule
			\end{tabular}
		\end{widetable}
	\end{table}

	\subsubsection{Three Dimensional Vector Potential Tests}
	\label{subsubsec:3d_vector_transfer}
	
	The three dimensional examples examine the discrete behavior of the $H(\curl)$ vector potential correction that replaces the planar scalar potential. Unless stated otherwise, the domain is
	\begin{equation}
		\Omega=[-1,1]^3,
		\qquad
		\Gamma=\partial\Omega ,
		\label{eq:3d_domain}
	\end{equation}
	and the manufactured solutions satisfy
	\begin{equation}
		-\nabla\cdot\bigl(\A^2(x)\nabla u(x)\bigr)=f(x)
		\quad \text{in }\Omega,
		\qquad
		u=g\quad \text{on }\Gamma .
		\label{eq:3d_model_problem}
	\end{equation}
	
	Here $\A$ denotes the symmetric square root of the physical diffusion tensor, so that the physical diffusion tensor is $\A^2$.
	
	Cases 1, 3, and 4 use the common polynomial exact solution
	\begin{equation}
		u_0=x_1^2+x_2^2+x_3^2+\frac{1}{10}x_1x_2x_3 .
		\label{eq:3d_u0}
	\end{equation}
	
	Case 2 uses a nonpolynomial Poisson solution so that the approximation error of the third weighted gradient recovery space is visible:
	\begin{equation}
		u_{\rm np}
		=
		1+
		\sin\frac{\pi x_1}{2}
		\sin\frac{\pi x_2}{2}
		\sin\frac{\pi x_3}{2},
		\qquad
		\A=I_3 .
		\label{eq:3d_nonpoly_solution}
	\end{equation}
	
	For this case,
	\begin{equation}
		f
		=
		\frac{3\pi^2}{4}
		\sin\frac{\pi x_1}{2}
		\sin\frac{\pi x_2}{2}
		\sin\frac{\pi x_3}{2}.
		\label{eq:3d_nonpoly_rhs}
	\end{equation}
	
	This test isolates the role of $\Eapp$: when the exact weighted gradient is not represented by the polynomial enrichment in the recovery space, the recovery space approximation error becomes visible in the final error split.
	
	Case 5 is a three dimensional flat interface jump problem. Let
	\[
	\Gamma_0=\{x_2=0\},\qquad
	\Omega_1=\{x_2<0\},\qquad
	\Omega_2=\{x_2>0\},
	\qquad
	n_0=(0,1,0)^T .
	\]
	
	We choose
	\begin{equation}
		\A=I_3 \quad \text{in } \Omega_1,
		\qquad
		\A=2I_3 \quad \text{in } \Omega_2 ,
		\label{eq:case5_coeff}
	\end{equation}
	so that the physical diffusion coefficient jumps from $1$ to $4$. The exact solution is prescribed piecewise by
	\begin{equation}
		u_2=\frac{x_1^2+x_2^2+x_3^2}{4},
		\qquad
		u_1=u_2+\frac{1+x_1+x_2+x_3}{4}.
		\label{eq:case5_piecewise_solution}
	\end{equation}
	
	Thus, on $\Gamma_0$,
	\begin{equation}
		\kappa_1=[u]_{\Gamma_0}
		=
		\frac{1+x_1+x_3}{4},
		\qquad
		\kappa_2=[\A^2\nabla u\cdot n_0]_{\Gamma_0}
		=
		\frac{1}{4}.
		\label{eq:case5_jump_data}
	\end{equation}
	
	The test includes both a nonzero solution jump and a nonzero normal flux jump. Since $\kappa_1$ varies along the interface, the interface tangential datum in the vector potential correction is also nonzero.
	
	Case 6 is an internal cube interface problem with
	\begin{equation}
		\Omega_1=\left[-\frac12,\frac12\right]^3,
		\qquad
		\Omega_2=\Omega\setminus\overline{\Omega}_1,
		\qquad
		\Gamma_0=\partial\Omega_1 .
		\label{eq:case6_internal_cube}
	\end{equation}
	
	The internal cube interface geometry used in Case~6 is shown in Figure~\ref{fig:internal_cube_domain}.
	
	\begin{figure}[htbp]
		\centering
		\includegraphics[width=0.62\textwidth]{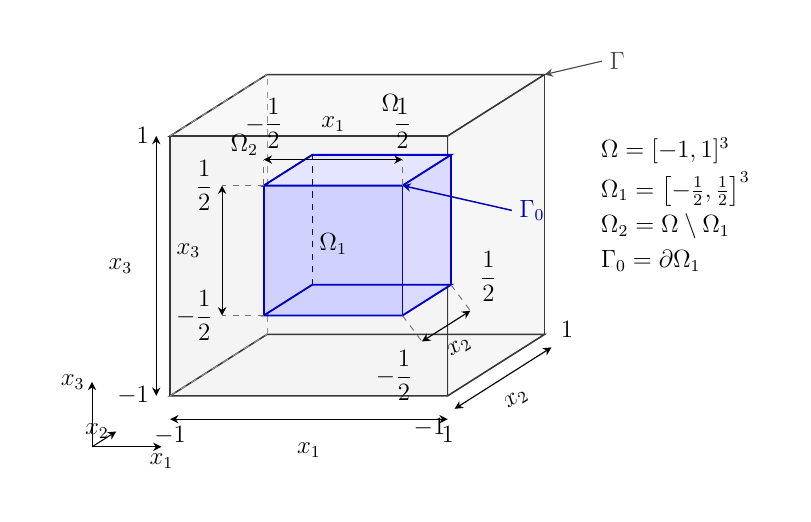}
		\caption{Cube interface geometry for Case~6 in \eqref{eq:case6_internal_cube} to \eqref{eq:case6_kappa2}. The outer domain is $\Omega=[-1,1]^3$, the inner subdomain is $\Omega_1=[-1/2,1/2]^3$, the exterior subdomain is $\Omega_2=\Omega\setminus\overline{\Omega}_1$, and the interface is $\Gamma_0=\partial\Omega_1$.}
		\label{fig:internal_cube_domain}
	\end{figure}
	
	We take
	\begin{equation}
		\A=2I_3 \quad \text{in } \Omega_1,
		\qquad
		\A=I_3 \quad \text{in } \Omega_2 .
		\label{eq:case6_coeff}
	\end{equation}
	
	Let
	\begin{equation}
		p=
		\frac{x_1^2+x_2^2+x_3^2+x_1x_2+x_2x_3+x_3x_1}{20},
		\qquad
		\chi=
		\frac{2+x_1+x_2+x_3}{10}.
		\label{eq:case6_p_chi}
	\end{equation}
	
	The piecewise exact solution is
	\begin{equation}
		u_1=p+\chi \quad \text{in } \Omega_1,
		\qquad
		u_2=p \quad \text{in } \Omega_2 .
		\label{eq:case6_piecewise_solution}
	\end{equation}
	
	Hence
	\begin{equation}
		\kappa_1=[u]_{\Gamma_0}=\chi|_{\Gamma_0}
		=
		\frac{2+x_1+x_2+x_3}{10},
		\qquad
		\nabla \kappa_1=\frac{1}{10}(1,1,1)^T .
		\label{eq:case6_kappa1}
	\end{equation}
	
	The corresponding interface flux jump is
	\begin{equation}
		\kappa_2
		=
		[\A^2\nabla u\cdot n_0]_{\Gamma_0}
		=
		\left(3\nabla p+\frac{2}{5}(1,1,1)^T\right)\cdot n_0 .
		\label{eq:case6_kappa2}
	\end{equation}
	
	This case tests a genuinely nonconstant solution jump on an internal three dimensional interface. The nodes in the inner cube are chosen to match the physical spacing of the outer point cloud, so the interface is not artificially over-resolved relative to the surrounding point cloud. The convergence fit is reported over $N_{\rm side}=7$ to $N_{\rm side}=19$, which is the double precision range used for the dense piecewise global MQ RBF prototype.
	
	Table~\ref{tab:3d_summary_compact_revised} gives the error ranges and fitted orders. Cases 1, 3, and 4 show monotone error reduction with refinement. Case 2 also decreases in both $L^2$ and $H^1_{\A}$, but it is included for a different purpose: it shows that $\Eapp$ can contribute substantially when the recovery space enrichment does not represent the exact weighted gradient.
	
	\begin{table}[!htbp]
		\centering
		\caption{Error ranges and fitted convergence orders for the three dimensional single domain problem \eqref{eq:3d_model_problem}, flat interface problem \eqref{eq:case5_piecewise_solution} to \eqref{eq:case5_jump_data}, and cube interface problem \eqref{eq:case6_piecewise_solution} to \eqref{eq:case6_kappa2}.}
		\label{tab:3d_summary_compact_revised}
		\small
		\begin{widetable}
			\begin{tabular}{c l c c c c c}
				\toprule
				Case & Type & $N_{\rm side}$ range
				& $L^2(\Omega)$ error range
				& $p_{L^2}^{\rm fit}$
				& $H^1_{\A}(\Omega)$ error range
				& $p_{H^1_{\A}}^{\rm fit}$ \\
				\midrule
				
				1 & Poisson
				& $11$ to $21$
				& $6.279{\times}10^{-4}\to3.234{\times}10^{-4}$
				& $0.95$
				& $2.105{\times}10^{-3}\to1.058{\times}10^{-3}$
				& $0.97$ \\
				
				2 & Nonpolynomial Poisson
				& $11$ to $21$
				& $8.789{\times}10^{-5}\to2.946{\times}10^{-5}$
				& $1.44$
				& $2.082{\times}10^{-3}\to1.079{\times}10^{-3}$
				& $0.84$ \\
				
				3 & Smooth variable coefficient
				& $11$ to $21$
				& $9.841{\times}10^{-4}\to2.278{\times}10^{-4}$
				& $2.12$
				& $2.112{\times}10^{-2}\to9.291{\times}10^{-3}$
				& $1.19$ \\
				
				4 & Nonsmooth variable coefficient
				& $11$ to $21$
				& $1.709{\times}10^{-3}\to6.259{\times}10^{-4}$
				& $1.47$
				& $3.383{\times}10^{-2}\to1.709{\times}10^{-2}$
				& $1.05$ \\
				
				5 & Flat interface jump
				& $11$ to $21$
				& $2.497{\times}10^{-2}\to1.172{\times}10^{-2}$
				& $1.12$
				& $5.764{\times}10^{-1}\to4.218{\times}10^{-1}$
				& $0.50$ \\
				
				6 & Internal interface with nonconstant $\kappa_1$
				& $7$ to $19$
				& $2.093{\times}10^{-2}\to9.777{\times}10^{-3}$
				& $0.79$
				& $1.191{\times}10^{-1}\to7.391{\times}10^{-2}$
				& $0.49$ \\
				
				\bottomrule
			\end{tabular}
		\end{widetable}
	\end{table}
	
	For the flat interface problem in Case 5, the $L^2$ error decreases overall, and the weighted energy error decreases with a lower fitted rate, with small intermediate fluctuations. This behavior is consistent with a larger projected transfer component in the interface setting. Case 6 uses a nonconstant jump on an internal interface. Over $N_{\rm side}=7$ to $N_{\rm side}=19$, both $L^2$ and $H^1_{\A}$ errors decrease. The fitted rates are lower than in the smooth single domain tests, reflecting the interface jump data, the piecewise recovery space, and the conditioning of the dense global RBF direct systems used in this realization.

	\begin{table}[!htbp]
		\centering
		\caption{Projected error decomposition and transfer response for the three dimensional tests \eqref{eq:3d_model_problem}, \eqref{eq:case5_piecewise_solution} to \eqref{eq:case5_jump_data}, and \eqref{eq:case6_piecewise_solution} to \eqref{eq:case6_kappa2}, using \eqref{eq:numerical_error_split} to \eqref{eq:total_bound_discrete}.}
		\label{tab:3d_transfer_compact_revised}
		\scriptsize
		\setlength{\tabcolsep}{3.8pt}
		\renewcommand{\arraystretch}{1.08}
		\begin{widetable}
			\begin{tabular}{c p{0.24\textwidth} c c c}
				\toprule
				Case & Type
				& $\Eapp/H^1_{\A}$ range
				& $\mathcal E_{\rm proj}^{\rm up}/H^1_{\A}$ behavior
				& Auxiliary transfer ratios \\
				\midrule
				
				1 & Poisson
				& $1.39{\times}10^{-7}$ to $4.40{\times}10^{-7}$
				& $1+O(10^{-7})$
				& $T_{123}=0.441,\quad T_{23}=0.906$ \\
				
				2 & Nonpolynomial Poisson
				& $0.470$ to $0.990$
				& $0.137$ to $0.883$
				& $T_{123}=0.481,\quad T_{23}=0.941$ \\
				
				3 & Smooth variable coefficient
				& $1.31{\times}10^{-6}$ to $3.83{\times}10^{-5}$
				& $1+O(10^{-5})$
				& $T_{123}=0.648,\quad T_{23}=0.951$ \\
				
				4 & Nonsmooth variable coefficient
				& $9.49{\times}10^{-7}$ to $2.69{\times}10^{-5}$
				& $1+O(10^{-5})$
				& $T_{123}=0.290,\quad T_{23}=0.922$ \\
				
				5 & Flat interface jump
				& $3.05{\times}10^{-9}$ to $6.54{\times}10^{-8}$
				& $1+O(10^{-8})$
				& $T_{123}=0.988,\quad T_{23}=0.986$ \\
				
				6 & Internal interface 
				& $8.14{\times}10^{-10}$ to $9.20{\times}10^{-6}$
				& $1+O(10^{-5})$
				& $T_{123}=0.648,\quad T_{23}=0.962$ \\
				
				\bottomrule
			\end{tabular}
		\end{widetable}
		\tablenote{The quantity $\mathcal E_{\rm proj}^{\rm up}$ is the projected combined upstream component entering the final recovery diagnostic.}
	\end{table}

	Table~\ref{tab:3d_transfer_compact_revised} reports numerical diagnostic values. The ratio $\Eapp/H^1_{\A}$ gives the relative size of the recovery space approximation error, and $\mathcal E_{\rm proj}^{\rm up}/H^1_{\A}$ gives the projected combined upstream component in \eqref{eq:numerical_error_split}. The auxiliary transfer ratios $T_{123}$ and $T_{23}$ record fixed space sensitivity: perturbations generated upstream and observed after recovery.
	
	In the polynomial manufactured solutions of Cases 1, 3, and 4, $\Eapp/H^1_{\A}$ is several orders of magnitude below one, so the weighted energy error is governed mainly by the projected transfer component. Case 1 is retained as the original polynomial Poisson diagnostic. Case 2 confirms the purpose of the split: for the nonpolynomial solution, the recovery space approximation error becomes visible and can dominate the final weighted gradient error. The two terms in \eqref{eq:numerical_error_split} are therefore both necessary for interpreting the final error.
	
	For the nonsmooth coefficient test in Case 4, the final error follows the projected component entering the recovery space. In the interface tests, $\mathcal E_{\rm proj}^{\rm up}$ is comparable with the final weighted energy error, whereas $\Eapp$ remains small. Case 5 serves as an interface stress test with a large projected component transmitted into the recovery space. In Case 6, the fixed space values $T_{123}=0.648$ and $T_{23}=0.962$ at $N_{\rm side}=15$ indicate partial attenuation of perturbations introduced in Subproblem~(S1), but much weaker attenuation when the perturbation is introduced directly in the Subproblem~(S2) correction field.
	
	The dense global MQ realization makes the sequential algebraic cost explicit. Table~\ref{tab:main_cost_summary} summarizes the largest reported three dimensional runs. The detailed record for each case, including solver choices and condition estimate definitions, is given in Appendix~\ref{app:validation_scale_cost}.
	
	\begingroup
	\begin{table}[!htbp]
		\centering
		\caption{Main text cost summary for the dense global RBF realization of NDM.}
		\label{tab:main_cost_summary}
		\small
		\begin{tabular}{lcccc}
			\toprule
			Test block
			& $N_{\rm side}$
			& Linear systems
			& DOF $(S1/S2/S3)$
			& CPU time (s) \\
			\midrule
			Single domain 3D cases 1 to 4
			& 21
			& 3
			& $9266/27794/9280$
			& $3970.77$ to $4866.37$ \\
			Flat interface case 5
			& 21
			& 4
			& $9266/27809/9740$
			& $2612.13$ \\
			Internal interface case 6
			& 19
			& 4
			& $6870/20588/7770$
			& $2080.72$ \\
			\bottomrule
		\end{tabular}
	\end{table}
	\endgroup
	
	These timings quantify the current dense realization. Table~\ref{tab:main_cost_summary} gives the cost profile of the dense global MQ realization, while scalable sparse or local realizations are separate implementation issues.
	
	The next subsection uses a two dimensional benchmark in the same noninterpolatory MQ RBF Galerkin setting to compare NDM, penalty, Nitsche, and multiplier boundary treatments in terms of accuracy, conditioning, boundary residuals, and parameter sensitivity.

	\subsection{Comparison with Classical Treatments of Essential Boundary Conditions}
	\label{subsec:ebc_mechanism_tests}
	
	This subsection compares boundary treatment mechanisms in a common noninterpolatory MQ RBF Galerkin space. NDM is represented by its fixed source, curl, and recovery transfer, while penalty, Nitsche, and Lagrange multiplier treatments provide standard weak or constrained alternatives. The comparison reports domain accuracy, boundary residual control, conditioning, and the effect of boundary-parameter or constraint choices. In the tables below, Strong pointwise denotes a nodal benchmark obtained by transforming the MQ basis to cardinal form and assigning the prescribed boundary values at selected boundary centers. All tests use the same nodes, kernel parameters, and integration rules. The algebraic conditioning diagnostic is denoted by $\Cond(K)$; for NDM it is the maximum condition estimate among the three sequential subproblem matrices.
	
	\subsubsection{FH Laplace Parameter Study}
	\label{subsubsec:unified_ebc_benchmark}
	
	The comparison uses the two dimensional Laplace benchmark of Fern\'andez-M\'endez and Huerta~\citep{FernandezMendezHuerta2004},
	\begin{equation}
		\Delta u=0
		\quad \text{in } \Omega=(0,1)^2,
		\label{eq:laplace_benchmark_pde}
	\end{equation}
	with
	\begin{equation}
		u(x,0)=\sin(\pi x),
		\qquad
		u(x,1)=u(0,y)=u(1,y)=0 .
		\label{eq:laplace_benchmark_bc}
	\end{equation}
	
	The exact solution is
	\begin{equation}
		u(x,y)=\frac{\sinh(\pi(1-y))}{\sinh(\pi)}\sin(\pi x).
		\label{eq:laplace_benchmark_exact}
	\end{equation}
	
	The absence of a forcing term and the presence of nontrivial boundary data make this benchmark useful for isolating the effect of the boundary treatment on both errors and algebraic conditioning.
	
	Penalty and Nitsche methods use $\beta_{\rm P}$ and $\beta_{\rm N}$ to control boundary enforcement, so the two parameter families are scanned separately. The penalty scaling follows Fern\'andez-M\'endez and Huerta. For Nitsche's method we include fixed diagnostic values and a trace generalized eigenvalue estimate computed from the present MQ RBF space, multiplied by a conservative safety factor. The Lagrange multiplier method uses boundary point constraints, while NDM is used as a parameter-free transfer mechanism.
	
	The fixed grid comparison for $N_{\rm side}=30$ and $c/h=3$ is organized in three summaries. Table~\ref{tab:fh_runtime_cost} reports the associated computational cost, separating the cost of one prescribed solve from the cost of choosing a boundary treatment. Table~\ref{tab:ebc_tradeoff_summary} extracts representative rows to display the tradeoff among accuracy, boundary residual, and conditioning. The complete parameter and strategy record is given in Table~\ref{tab:ebc_selected_summary}, where preselected choices are distinguished from post hoc $H^1$ optima. Rows marked as post hoc $H^1$ optima use the exact solution to minimize $|e|_{H^1(\Omega)}$ over the scanned list and are diagnostic lower envelope rows for the scanned families, not a priori choices.
	
	With post hoc $H^1$ optimum parameter selection, penalty and Nitsche give the smallest domain errors in this benchmark: their $H^1$ seminorm errors are $1.372\times10^{-2}$ and $1.364\times10^{-2}$, respectively, compared with $4.065\times10^{-2}$ for the fixed NDM transfer. These rows show the accuracy attainable when exact solution information selects boundary parameter values from the scanned families. The NDM row provides a parameter-free transfer mechanism in the same MQ space.
	
	The cost comparison separates a prescribed solve from the choice of boundary treatment parameters. NDM uses the fixed sequence S1, S2, and S3. Penalty and Nitsche methods use one scalar solve after a boundary parameter has been prescribed, while a scan requires one solve per candidate. The comparison therefore reports accuracy, conditioning, and parameter selection separately.
	
	All four methods use the same noninterpolatory MQ RBF space, so the table compares boundary terms, multiplier constraints, and the sequential natural transfer mechanism within a common approximation setting.
	
	\begingroup
	\begin{table}[!htbp]
		\centering
		\caption{Wall clock cost on the two dimensional FH benchmark \eqref{eq:laplace_benchmark_pde} to \eqref{eq:laplace_benchmark_exact}, with fixed $N_{\rm side}=30$ and $c/h=3$. The common MQ data construction time was $0.610$ s and is not assigned to any individual method. Condition number diagnostics and plotting are excluded.}
		\label{tab:fh_runtime_cost}
		\scriptsize
		\setlength{\tabcolsep}{3.2pt}
		\begin{widetable}
			\begin{tabular}{llccccc}
				\toprule
				Method
				& Boundary choice
				& Search size
				& Setup time (s)
				& Method time (s)
				& Total time (s)
				& $|e|_{H^1(\Omega)}$ \\
				\midrule
				Penalty
				& one prescribed $\beta_{\rm P}$
				& 1
				& $0.000$
				& $0.028$
				& $0.028$
				& $6.017{\times}10^{-1}$ \\
				Penalty
				& diagnostic scan
				& 19
				& $0.000$
				& $0.491$
				& $0.491$
				& $1.372{\times}10^{-2}$ \\
				Nitsche
				& trace eigenvalue
				& 1
				& $0.141$
				& $0.029$
				& $0.170$
				& $7.650{\times}10^{-2}$ \\
				Nitsche
				& diagnostic scan
				& 19
				& $0.000$
				& $0.533$
				& $0.533$
				& $1.364{\times}10^{-2}$ \\
				Lagrange multiplier
				& fixed constraint set
				& 1
				& $0.000$
				& $0.053$
				& $0.053$
				& $2.319{\times}10^{-2}$ \\
				NDM
				& fixed natural transfer
				& 1
				& $0.000$
				& $0.531$
				& $0.531$
				& $4.065{\times}10^{-2}$ \\
				\bottomrule
			\end{tabular}
		\end{widetable}
		\tablenote{The setup time records method-specific parameter construction, such as the trace eigenvalue estimate for the Nitsche row. The method time records the wall clock time of the corresponding solve or scan after the common MQ data are available. The diagnostic scan rows report the best $H^1$ row selected from the scanned list using exact solution information.}
	\end{table}
	\endgroup
	
	Table~\ref{tab:fh_runtime_cost} gives the wall clock cost associated with these choices. A single prescribed penalty or Nitsche solve requires less time than the three-step NDM transfer, while the diagnostic scans have costs comparable to the fixed NDM sequence in this benchmark.
	
	Table~\ref{tab:ebc_tradeoff_summary} extracts representative preselected and post hoc $H^1$ optimum rows from Table~\ref{tab:ebc_selected_summary} to show the accompanying conditioning tradeoff.
	
	\begin{table}[!htbp]
		\centering
		\caption{Representative accuracy and conditioning tradeoff for the FH Laplace benchmark \eqref{eq:laplace_benchmark_pde} to \eqref{eq:laplace_benchmark_exact}, with $N_{\rm side}=30$ and $c/h=3$.}
		\label{tab:ebc_tradeoff_summary}
		\small
		\begin{tabular}{lcccc}
			\toprule
			Method & Parameter choice & $|e|_{H^1(\Omega)}$ & $L^2(\Gamma)$ & $\Cond(K)$ \\
			\midrule
			Penalty
			& FH scaling, preselected
			& $6.017{\times}10^{-1}$
			& $4.02{\times}10^{-5}$
			& $1.395{\times}10^{20}$ \\
			Penalty
			& post hoc $H^1$ optimum
			& $1.372{\times}10^{-2}$
			& $1.443{\times}10^{-4}$
			& $8.512{\times}10^{19}$ \\
			Nitsche
			& Trace eigenvalue, preselected
			& $7.650{\times}10^{-2}$
			& $2.788{\times}10^{-3}$
			& $1.957{\times}10^{18}$ \\
			Nitsche
			& post hoc $H^1$ optimum
			& $1.364{\times}10^{-2}$
			& $2.611{\times}10^{-5}$
			& $2.024{\times}10^{19}$ \\
			Lagrange
			& Boundary node multipliers
			& $2.319{\times}10^{-2}$
			& $2.213{\times}10^{-4}$
			& $1.403{\times}10^{15}$ \\
			NDM
			& None
			& $4.065{\times}10^{-2}$
			& $1.647{\times}10^{-3}$
			& $1.176{\times}10^{15}$ \\
			\bottomrule
		\end{tabular}
		\tablenote{Rows labelled post hoc $H^1$ optimum use the exact solution only to select a candidate from the scanned parameter list.}
	\end{table}
	
	\begin{figure}[htbp]
		\centering
		\includegraphics[width=0.5\textwidth]{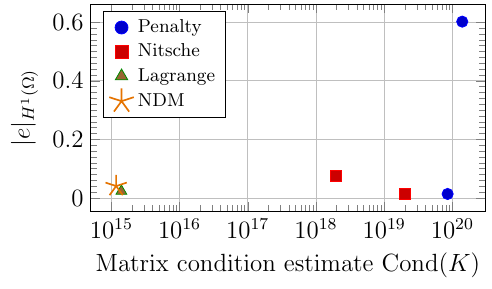}
		\caption{Accuracy and conditioning scatter for the FH Laplace benchmark \eqref{eq:laplace_benchmark_pde} to \eqref{eq:laplace_benchmark_exact}, based on the representative rows in Table~\ref{tab:ebc_tradeoff_summary}. The NDM point is shown by a larger star marker and represents the fixed natural transfer reference; penalty and Nitsche points reflect selected parameter strategies.}
		\label{fig:ebc_tradeoff_pareto}
	\end{figure}
	
	Figure~\ref{fig:ebc_tradeoff_pareto} is a diagnostic Pareto view of the accuracy and conditioning tradeoff. A point improves on another in this display only if it reduces both the domain error and the condition estimate. The post hoc penalty and Nitsche points show lower domain errors after boundary parameter selection by the exact solution. The multiplier point gives a smaller $H^1$ error in this benchmark together with a saddle point formulation and a chosen constraint set. The NDM point identifies the parameter-free natural transfer profile within the same MQ approximation setting.
	
	\begin{table}[!htbp]
		\centering
		\caption{Parameter strategies, errors, and matrix conditioning for the FH Laplace benchmark \eqref{eq:laplace_benchmark_pde} to \eqref{eq:laplace_benchmark_exact}, with fixed $N_{\rm side}=30$ and $c/h=3$.}
		\label{tab:ebc_selected_summary}
		\scriptsize
		\begin{widetable}
			\begin{tabular}{llcccccc}
				\toprule
				Method & Parameter strategy & Post hoc $H^1$ optimum & Parameter value & $L^2(\Omega)$ & $|e|_{H^1(\Omega)}$ & $L^2(\Gamma)$ & $\Cond(K)$ \\
				\midrule
				Penalty
				& FH scaling $(10^4/8)h^{-2}$
				& No
				& $1.051{\times}10^6$
				& $8.100{\times}10^{-3}$
				& $6.017{\times}10^{-1}$
				& $4.02{\times}10^{-5}$
				& $1.395{\times}10^{20}$ \\
				Penalty
				& $10^4h^{-1}$
				& No
				& $2.900{\times}10^5$
				& $2.212{\times}10^{-3}$
				& $1.781{\times}10^{-1}$
				& $1.80{\times}10^{-5}$
				& $6.174{\times}10^{19}$ \\
				Penalty
				& post hoc $H^1$ optimum
				& Yes
				& $1.000{\times}10^4$
				& $1.330{\times}10^{-4}$
				& $1.372{\times}10^{-2}$
				& $1.443{\times}10^{-4}$
				& $8.512{\times}10^{19}$ \\
				\midrule
				Nitsche
				& Current RBF trace eigenvalue
				& No
				& $1.748{\times}10^2$
				& $5.13{\times}10^{-4}$
				& $7.650{\times}10^{-2}$
				& $2.788{\times}10^{-3}$
				& $1.957{\times}10^{18}$ \\
				Nitsche
				& Fixed diagnostic value $\beta_{\rm N}=20$
				& No
				& $20$
				& $5.33{\times}10^{-4}$
				& $4.417{\times}10^{-2}$
				& $2.759{\times}10^{-3}$
				& $9.814{\times}10^{16}$ \\
				Nitsche
				& post hoc $H^1$ optimum
				& Yes
				& $1.000{\times}10^4$
				& $1.284{\times}10^{-4}$
				& $1.364{\times}10^{-2}$
				& $2.611{\times}10^{-5}$
				& $2.024{\times}10^{19}$ \\
				\midrule
				Lagrange
				& Boundary node multipliers
				& No
				& n.a.
				& $2.052{\times}10^{-4}$
				& $2.319{\times}10^{-2}$
				& $2.213{\times}10^{-4}$
				& $1.403{\times}10^{15}$ \\
				NDM
				& $S_1$ to $S_2$ to $S_3$
				& No
				& n.a.
				& $6.363{\times}10^{-4}$
				& $4.065{\times}10^{-2}$
				& $1.647{\times}10^{-3}$
				& $1.176{\times}10^{15}$ \\
				\bottomrule
			\end{tabular}
		\end{widetable}
		\tablenote{The post hoc $H^1$ optimum rows use the exact solution only to select a representative scanned parameter value.}
	\end{table}
	
	The parameter scans make the tradeoff more explicit. In Figure~\ref{fig:ebc_beta_sweep}, the penalty and Nitsche domain errors vary over clear parameter windows; increasing the boundary enforcement strength alone does not guarantee a smaller domain error. Figure~\ref{fig:ebc_condition_raw} shows the corresponding change in $\Cond(K)$. Since the four formulations lead to different algebraic systems, these condition estimates are read as mechanism-specific conditioning diagnostics within the common global MQ RBF trial space.
	
	\begin{figure}[htbp]
		\centering
		\includegraphics[width=0.90\textwidth]{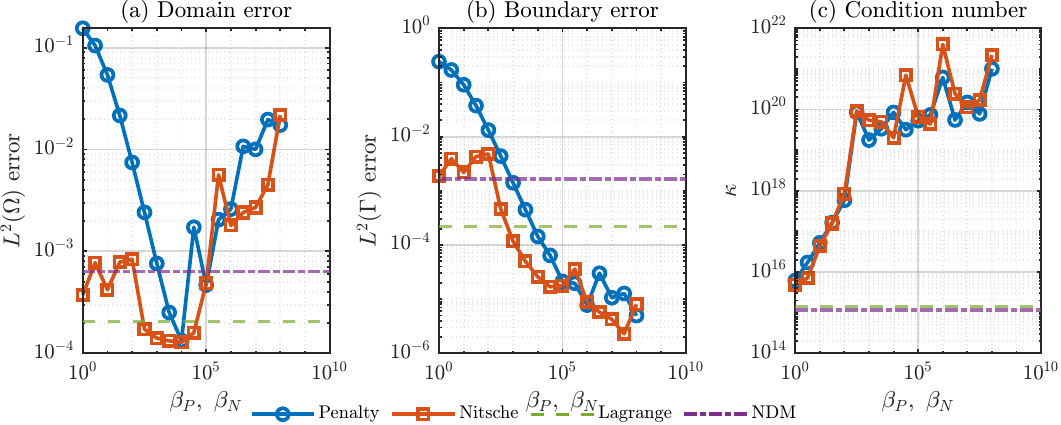}
		\caption{Boundary parameter scan for the FH Laplace benchmark \eqref{eq:laplace_benchmark_pde} to \eqref{eq:laplace_benchmark_exact}, with fixed $N_{\rm side}=30$ and $c/h=3$. Penalty varies with $\beta_{\rm P}$ and Nitsche varies with $\beta_{\rm N}$; the Lagrange multiplier method and NDM are reference lines without boundary parameter scans.}
		\label{fig:ebc_beta_sweep}
	\end{figure}
	
	\begin{figure}[htbp]
		\centering
		\includegraphics[width=0.90\textwidth]{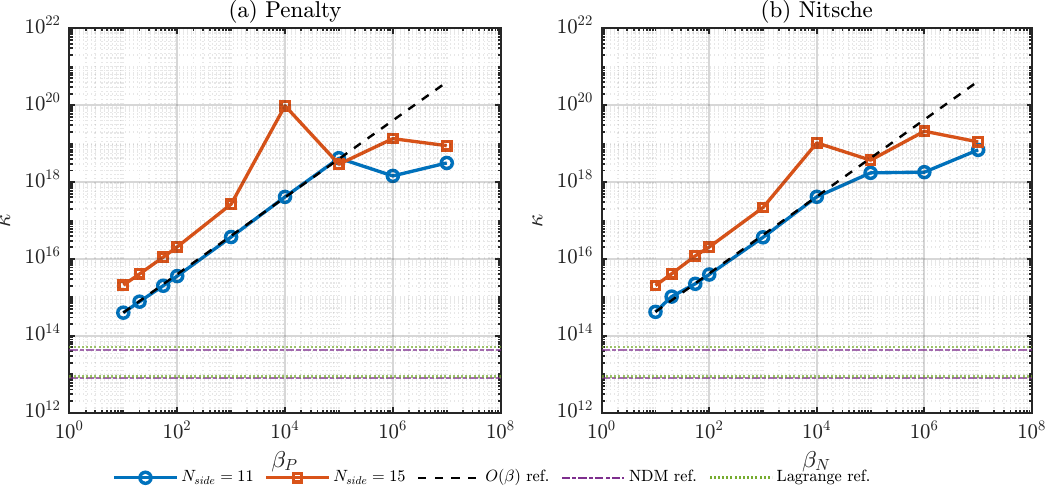}
		\caption{Condition estimate $\Cond(K)$ versus boundary parameter for the FH Laplace benchmark \eqref{eq:laplace_benchmark_pde} to \eqref{eq:laplace_benchmark_exact}. The horizontal NDM reference is the maximum condition estimate among the three sequential subproblem matrices \eqref{eq:3d-unified-s1} to \eqref{eq:3d-unified-s3}.}
		\label{fig:ebc_condition_raw}
	\end{figure}
	
	Table~\ref{tab:ndm_subproblem_condition} separates the condition estimates of the three NDM subproblems. All three increase under refinement, reflecting the inherent conditioning of dense global MQ RBF spaces. The decomposition acts as a boundary transfer mechanism implemented by natural subproblems, and the conditioning reflects the dense global MQ realization. 
	
	\begin{table}[!htbp]
		\centering
		\caption{Condition estimate decomposition for the three NDM subproblems \eqref{eq:3d-unified-s1} to \eqref{eq:3d-unified-s3}, evaluated on the FH Laplace benchmark \eqref{eq:laplace_benchmark_pde} to \eqref{eq:laplace_benchmark_exact}.}
		\label{tab:ndm_subproblem_condition}
		\small
		\begin{tabular}{cccccc}
			\toprule
			$N_{\rm side}$ & $h$ & $\Cond_{S_1}$ & $\Cond_{S_2}$ & $\Cond_{S_3}$ & $\Cond_{\max}$ \\
			\midrule
			11 & 0.1000 & $7.975{\times}10^{12}$ & $7.969{\times}10^{12}$ & $5.191{\times}10^{12}$ & $7.975{\times}10^{12}$ \\
			15 & 0.0714 & $4.255{\times}10^{13}$ & $4.254{\times}10^{13}$ & $3.448{\times}10^{13}$ & $4.255{\times}10^{13}$ \\
			21 & 0.0500 & $2.271{\times}10^{14}$ & $2.271{\times}10^{14}$ & $2.064{\times}10^{14}$ & $2.271{\times}10^{14}$ \\
			30 & 0.0345 & $1.176{\times}10^{15}$ & $1.176{\times}10^{15}$ & $1.128{\times}10^{15}$ & $1.176{\times}10^{15}$ \\
			\bottomrule
		\end{tabular}
	\end{table}
	
	Figures~\ref{fig:ebc_convergence_penalty} and~\ref{fig:ebc_convergence_nitsche} show the corresponding refinement curves. The curves agree with the fixed grid scans: penalty and Nitsche methods can give small errors when the parameters are well chosen, but the effective choices lie in parameter windows. NDM gives the corresponding natural transfer curve for the essential boundary data.
	
	\begin{figure}[t]
		\centering
		\includegraphics[width=0.92\textwidth]{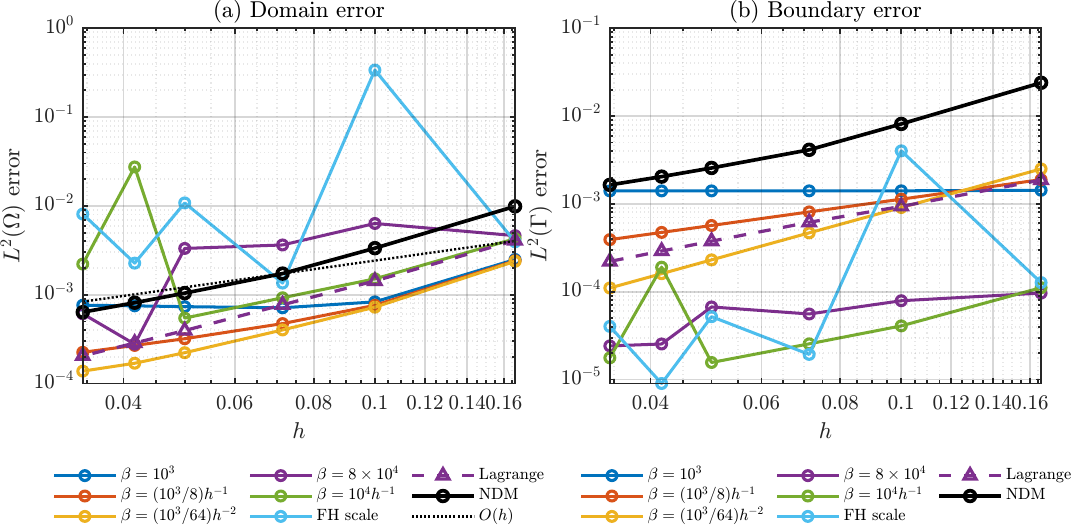}
		\caption{Refinement curves for penalty boundary treatments on the FH Laplace benchmark \eqref{eq:laplace_benchmark_pde} to \eqref{eq:laplace_benchmark_exact}. Penalty parameter strategies are compared with the Lagrange multiplier and NDM reference curves, which do not use a scanned boundary parameter.}
		\label{fig:ebc_convergence_penalty}
	\end{figure}
	
	\pagebreak[3]
	
	\begin{figure}[t]
		\centering
		\includegraphics[width=0.92\textwidth]{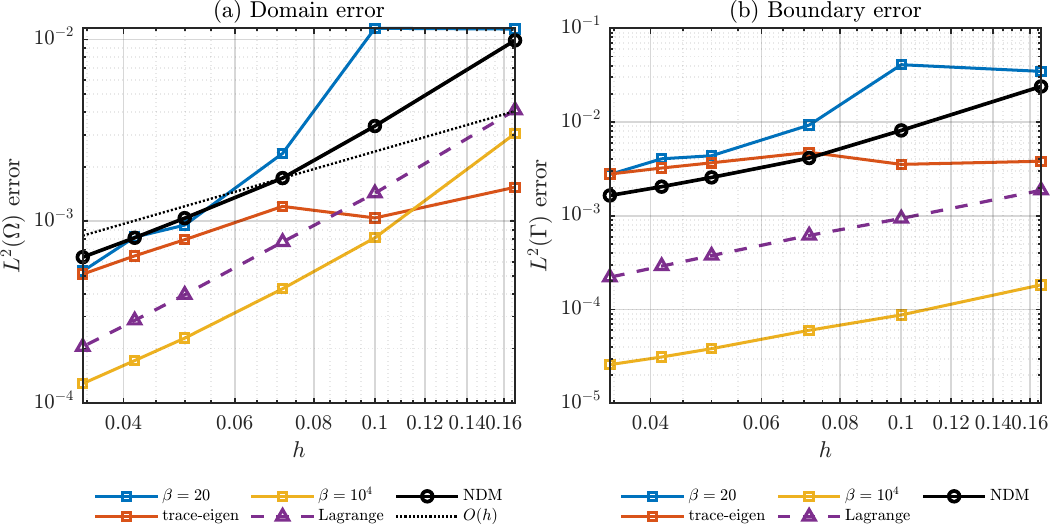}
		\caption{Refinement curves for Nitsche boundary treatments on the FH Laplace benchmark \eqref{eq:laplace_benchmark_pde} to \eqref{eq:laplace_benchmark_exact}. Stabilization strategies are compared with the Lagrange multiplier and NDM reference curves, which do not use a scanned boundary parameter.}
		\label{fig:ebc_convergence_nitsche}
	\end{figure}
	
	\subsubsection{L-Shaped Singular Benchmark}
	\label{subsubsec:lshape_singular}
	
	The square benchmark isolates boundary parameter effects on a smooth domain. We also test an L-shaped domain to examine the natural decomposition transfer in the presence of a reentrant corner. Such domains are standard in meshfree and partition of unity studies of essential boundary condition treatments. Griebel and Schweitzer~\citep{GriebelSchweitzer2003} and Schweitzer~\citep{Schweitzer2009}, for example, used L-shaped domains to assess Nitsche and algebraic conforming treatments. In these tests, the harmonic singularity limits convergence, while graded point distributions near the corner improve the observed rates.
	
	We consider
	\begin{equation}
		\Omega_L=(-1,1)^2\setminus [0,1]^2 ,
		\label{eq:lshape_domain}
	\end{equation}
	and prescribe the exact singular solution
	\begin{equation}
		u(r,\theta)
		=
		r^{2/3}
		\sin\left(\frac{2\theta-\pi}{3}\right),
		\label{eq:lshape_singular_solution}
	\end{equation}
	where $(r,\theta)$ are polar coordinates centered at the reentrant corner. The manufactured problem is
	\begin{equation}
		-\Delta u=0
		\quad \text{in } \Omega_L,
		\qquad
		u=g
		\quad \text{on } \partial\Omega_L ,
		\label{eq:lshape_problem}
	\end{equation}
	with $g=u|_{\partial\Omega_L}$. The test uses the same singular profile as the meshfree essential boundary benchmarks just cited, while using the present MQ RBF NDM discretization and boundary transfer mechanism. It checks the behavior of the decomposition in the regularity limited regime induced by the reentrant corner. Since
	\[
	|\nabla u|\sim r^{-1/3}
	\quad \text{near the reentrant corner},
	\]
	uniform centers are expected to produce lower convergence rates than smooth manufactured solutions. Following the diagnostic convention used by Schweitzer~\citep{Schweitzer2009}, we report rates with respect to the number of degrees of freedom, denoted by $\rho$. For uniform centers, the rates expected when convergence is limited by the singularity are approximately
	\[
	\rho_{L^2}\approx \frac{2}{3},
	\qquad
	\rho_{H^1}\approx \frac{1}{3}.
	\]
	
	With mild local grading near the reentrant corner, the rates are expected to approach the graded point behavior,
	\[
	\rho_{L^2}\approx 1,
	\qquad
	\rho_{H^1}\approx \frac{1}{2}.
	\]
	
	Table~\ref{tab:lshape_fit_rates} gives the fitted rates, and Figure~\ref{fig:lshape_uniform_graded} compares uniform and mildly graded centers. Uniform centers recover the trend limited by regularity expected under a corner singularity. Mild grading improves both error norms: the $H^1$ rate increases from about $0.36$ to $0.56$, and the $L^2$ rate from about $0.73$ to $1.22$. The diagnostic rate of $\Eapp$ also increases, from about $0.36$ to $0.54$, indicating that the improvement is mainly due to better resolution of the singular weighted gradient near the reentrant corner.
	
	\begin{table}[!htbp]
		\centering
		\caption{Fitted degrees of freedom convergence rates for the L-shaped singular benchmark \eqref{eq:lshape_domain} to \eqref{eq:lshape_problem}.}
		\label{tab:lshape_fit_rates}
		\small
		\begin{tabular}{lcccc}
			\toprule
			Center distribution
			& $\rho_{L^2}$
			& $\rho_{H^1}$
			& $\rho_{\rm APP}$
			& $\rho_{\rm S2proj}$ \\
			\midrule
			Uniform centers
			& $0.728$
			& $0.357$
			& $0.357$
			& $0.360$ \\
			Mild graded centers
			& $1.218$
			& $0.559$
			& $0.541$
			& $0.665$ \\
			\bottomrule
		\end{tabular}
	\end{table}
	
	\begin{figure}[htbp]
		\centering
		\includegraphics[width=0.78\textwidth]{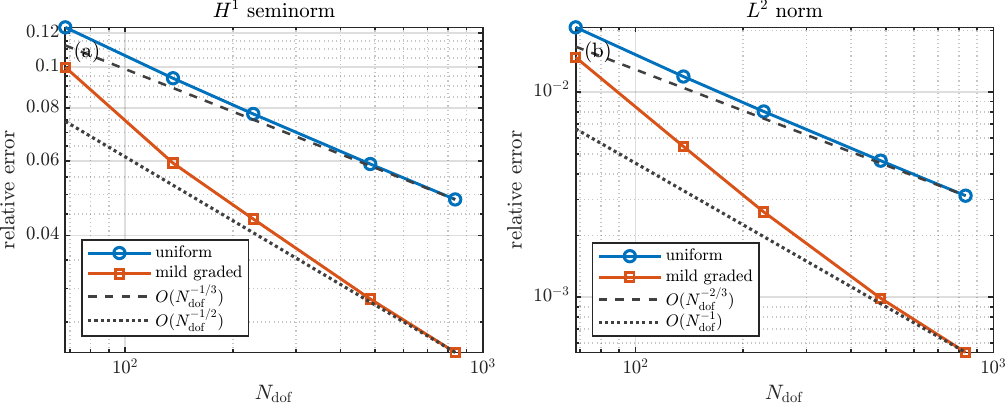}
		\caption{NDM convergence on the L-shaped singular benchmark \eqref{eq:lshape_domain} to \eqref{eq:lshape_problem}. Uniform centers recover rates expected under a corner singularity, while mild local grading near the reentrant corner improves the observed degrees of freedom rates.}
		\label{fig:lshape_uniform_graded}
	\end{figure}
	\subsubsection{Constraint Density and Point Clouds}
	\label{subsubsec:lagrange_scattered}
	
	The multiplier method introduces a different discrete choice: the multiplier constraint density. Table~\ref{tab:lagrange_density} shows its effect in the same RBF space. With too few constraint points, both domain and boundary errors are large. Boundary node constraints give a better balance. Doubling the boundary constraint density reduces the boundary error to $2.510\times10^{-8}$, but increases the domain $H^1$ error to $7.969\times10^{-2}$. Thus a very small boundary point residual does not by itself imply the best domain solution; the multiplier pathway is sensitive to the constraint space.
	
	\begin{table}[!htbp]
		\centering
		\caption{Sensitivity of the Lagrange multiplier method to boundary constraint point density for the FH Laplace benchmark \eqref{eq:laplace_benchmark_pde} to \eqref{eq:laplace_benchmark_exact}, with $N_{\rm side}=30$ and $c/h=3$.}
		\label{tab:lagrange_density}
		\small
		\begin{tabular}{lcccc}
			\toprule
			Constraint point distribution & Number of constraints & $L^2(\Omega)$ & $|e|_{H^1(\Omega)}$ & $L^2(\Gamma)$ \\
			\midrule
			half boundary nodes
			& 60
			& $3.236{\times}10^{-3}$
			& $1.519{\times}10^{-1}$
			& $1.024{\times}10^{-2}$ \\
			boundary nodes
			& 116
			& $2.052{\times}10^{-4}$
			& $2.319{\times}10^{-2}$
			& $2.213{\times}10^{-4}$ \\
			double boundary nodes
			& 232
			& $1.056{\times}10^{-3}$
			& $7.969{\times}10^{-2}$
			& $2.510{\times}10^{-8}$ \\
			\bottomrule
		\end{tabular}
	\end{table}
	
	We next replace the tensor grid by quasi-uniform jittered point clouds and a boundary enriched point cloud. Table~\ref{tab:scattered_node_cloud} reports the results for $N_{\rm side}=21$, with means and standard deviations over five seeds for the jittered cases. For NDM, the $H^1$ standard deviations are $1.71\times10^{-4}$ under $0.15h$ jitter and $3.87\times10^{-4}$ under $0.30h$ jitter, and no solver failure occurs. Boundary enrichment reduces the NDM $H^1$ error from $6.077\times10^{-2}$ on tensor product nodes to $5.731\times10^{-2}$. These data show that the NDM RBF implementation also operates on quasi-uniform point clouds, and that the tested perturbations have a small effect on the reported NDM errors. The comparison with the multiplier rows shows that NDM uses the same transfer mechanism on these point clouds, while multiplier behavior depends on constraint density.
	
	\begin{table}[!htbp]
		\centering
		\caption{Point cloud stability of essential boundary condition treatments for the FH Laplace benchmark \eqref{eq:laplace_benchmark_pde} to \eqref{eq:laplace_benchmark_exact}, with $N_{\rm side}=21$ and $c/h=3$.}
		\label{tab:scattered_node_cloud}
		\scriptsize
		\begin{widetable}
			\begin{tabular}{llcccc}
				\toprule
				Point cloud & Method & $N_{\rm seed}$ & $L^2(\Omega)$ & $|e|_{H^1(\Omega)}$ & $L^2(\Gamma)$ \\
				\midrule
				tensor
				& Penalty & 1 & $4.168{\times}10^{-3}$ & $2.406{\times}10^{-1}$ & $3.60{\times}10^{-5}$ \\
				tensor
				& Nitsche & 1 & $4.76{\times}10^{-4}$ & $5.775{\times}10^{-2}$ & $2.430{\times}10^{-3}$ \\
				tensor
				& Lagrange & 1 & $3.44{\times}10^{-4}$ & $2.699{\times}10^{-2}$ & $3.33{\times}10^{-4}$ \\
				tensor
				& NDM & 1 & $1.221{\times}10^{-3}$ & $6.077{\times}10^{-2}$ & $3.239{\times}10^{-3}$ \\
				\midrule
				jitter $0.15h$
				& Penalty & 5 & $(5.277\pm5.101){\times}10^{-3}$ & $(3.034\pm2.939){\times}10^{-1}$ & $(6.7\pm7.3){\times}10^{-5}$ \\
				jitter $0.15h$
				& Nitsche & 5 & $(4.82\pm0.07){\times}10^{-4}$ & $(5.884\pm0.103){\times}10^{-2}$ & $(2.480\pm0.047){\times}10^{-3}$ \\
				jitter $0.15h$
				& Lagrange & 5 & $(3.61\pm0.04){\times}10^{-4}$ & $(2.863\pm0.046){\times}10^{-2}$ & $(4.76\pm0.47){\times}10^{-4}$ \\
				jitter $0.15h$
				& NDM & 5 & $(1.227\pm0.005){\times}10^{-3}$ & $(6.099\pm0.017){\times}10^{-2}$ & $(3.249\pm0.011){\times}10^{-3}$ \\
				\midrule
				jitter $0.30h$
				& Penalty & 5 & $(4.663\pm4.435){\times}10^{-3}$ & $(2.660\pm2.511){\times}10^{-1}$ & $(5.2\pm5.0){\times}10^{-5}$ \\
				jitter $0.30h$
				& Nitsche & 5 & $(5.13\pm0.16){\times}10^{-4}$ & $(6.307\pm0.222){\times}10^{-2}$ & $(2.677\pm0.106){\times}10^{-3}$ \\
				jitter $0.30h$
				& Lagrange & 5 & $(3.84\pm0.09){\times}10^{-4}$ & $(3.022\pm0.092){\times}10^{-2}$ & $(5.92\pm0.69){\times}10^{-4}$ \\
				jitter $0.30h$
				& NDM & 5 & $(1.243\pm0.011){\times}10^{-3}$ & $(6.153\pm0.039){\times}10^{-2}$ & $(3.271\pm0.023){\times}10^{-3}$ \\
				\midrule
				boundary enriched
				& Penalty & 1 & $1.731{\times}10^{-3}$ & $1.010{\times}10^{-1}$ & $2.7{\times}10^{-5}$ \\
				boundary enriched
				& Nitsche & 1 & $1.416{\times}10^{-3}$ & $1.908{\times}10^{-1}$ & $8.465{\times}10^{-3}$ \\
				boundary enriched
				& Lagrange & 1 & $5.54{\times}10^{-4}$ & $4.387{\times}10^{-2}$ & $1.449{\times}10^{-3}$ \\
				boundary enriched
				& NDM & 1 & $1.129{\times}10^{-3}$ & $5.731{\times}10^{-2}$ & $2.961{\times}10^{-3}$ \\
				\bottomrule
			\end{tabular}
		\end{widetable}
		\tablenote{Jittered point clouds use five random seeds; values in parentheses are standard deviations.}
	\end{table}
	\subsection{Boundary Perturbation Response}
	\label{subsec:perturbation_transfer}
	
	The preceding tests compare methods under unperturbed boundary data. We next add a small essential boundary component on part of the boundary and measure its recovered interior propagation together with its fidelity to the harmonic reference generated by the same boundary input. This design separates attenuation by the discrete recovery space from agreement with the continuous harmonic response. We therefore keep the same smooth Poisson problem and add a controlled boundary component. Let
	\begin{equation}
		\Omega=(-1,1)^2,
		\qquad
		u(x,y)=x^2+y^2+\sin(x+y),
		\qquad
		-\Delta u=-4+2\sin(x+y),
		\label{eq:perturbation_base_problem}
	\end{equation}
	and add a perturbation only on the bottom side of the unperturbed boundary data $g$:
	\[
	\Gamma_b=\{(x,-1):-1\le x\le 1\}.
	\]
	
	The perturbed data are
	\begin{equation}
		g_\varepsilon=g+\varepsilon\eta_k,
		\qquad
		\varepsilon=10^{-2},
		\label{eq:perturbed_boundary_data}
	\end{equation}
	where
	\begin{equation}
		\eta_k(x,-1)
		=
		(1-x^2)^2
		\sin\!\left(\frac{k\pi(x+1)}{2}\right),
		\qquad
		\eta_k=0\quad \text{on }\Gamma\setminus\Gamma_b .
		\label{eq:boundary_perturbation_mode}
	\end{equation}
	
	The factor $(1-x^2)^2$ makes the perturbation vanish at the corners, avoiding an additional corner incompatibility. We test $k=4,8,12,16$.
	
	For any method $M$, define the incremental response by
	\begin{equation}
		\Delta u_h^M
		=
		u_h^M(g+\varepsilon\eta_k)-u_h^M(g).
		\label{eq:incremental_response_def}
	\end{equation}
	
	This subtraction removes the baseline discretization error associated with the unperturbed data and isolates the response to the added boundary input. We also construct the harmonic reference for the bottom side perturbation, so that the incremental response can be compared with the continuous essential-boundary response. Let
	\[
	X=\frac{x+1}{2},
	\qquad
	Y=\frac{y+1}{2} .
	\]
	
	The harmonic extension of the continuous essential boundary perturbation problem can be written as
	\begin{equation}
		\Delta u_{\rm ref}(x,y)
		=
		\varepsilon\sum_{n=1}^{N_m}
		b_n\sin(n\pi X)
		\frac{\sinh(n\pi(1-Y))}{\sinh(n\pi)},
		\label{eq:harmonic_reference_increment}
	\end{equation}
	where $N_m$ is the number of sine modes retained in the numerical computation, and
	\begin{equation}
		b_n
		=
		2\int_0^1
		\left[1-(2s-1)^2\right]^2
		\sin(k\pi s)\sin(n\pi s)\,ds .
		\label{eq:harmonic_reference_coeff}
	\end{equation}
	
	Define the thin layer adjacent to the perturbed boundary by
	\begin{equation}
		\Omega_\rho
		=
		\{(x,y)\in\Omega:-1\le y\le -1+\rho\},
		\qquad
		\rho=\max\{3h,0.15\},
		\label{eq:strip_region_def}
	\end{equation}
	and take the near boundary interior observation line $y_\rho=-1+2h$. Normalized by the $L^2(\Gamma)$ norm of the input perturbation, the thin layer propagation gain, interior propagation gain, observation-line gain, and near boundary total variation gain are defined as
	\begin{equation}
		\begin{aligned}
			G_{\rm strip}^M(k)\mathrel{:=}
			\frac{\|\Delta u_h^M\|_{L^2(\Omega_\rho)}}{\|\varepsilon\eta_k\|_{L^2(\Gamma)}};\;
			G_{\rm int}^M(k)\mathrel{:=}
			\frac{\|\Delta u_h^M\|_{L^2(\Omega\setminus\Omega_\rho)}}{\|\varepsilon\eta_k\|_{L^2(\Gamma)}};\\
			G_{\rm line}^M(k)\mathrel{:=}
			\frac{\|\Delta u_h^M(\cdot,y_\rho)\|_{L^2(-1,1)}}{\|\varepsilon\eta_k\|_{L^2(\Gamma)}};\;
			G_{\rm TV}^M(k)\mathrel{:=}
			\frac{\operatorname{TV}(\Delta u_h^M(\cdot,y_\rho))}{\|\varepsilon\eta_k\|_{L^2(\Gamma)}} .
		\end{aligned}
		\label{eq:gains}
	\end{equation}
	
	Here $\operatorname{TV}$ is the discrete total variation along $y=y_\rho$, used as a measure of near boundary oscillation. Relative errors with respect to the harmonic reference are defined by
	\begin{equation}
		E_{\rm ref,L^2}^M(k)
		=
		\frac{\|\Delta u_h^M-\Delta u_{\rm ref}\|_{L^2(\Omega)}}
		{\|\Delta u_{\rm ref}\|_{L^2(\Omega)}},
		\qquad
		E_{\rm ref,H^1}^M(k)
		=
		\frac{|\Delta u_h^M-\Delta u_{\rm ref}|_{H^1(\Omega)}}
		{|\Delta u_{\rm ref}|_{H^1(\Omega)}} .
		\label{eq:reference_relative_error}
	\end{equation}
	
	For NDM, we also record the projection transfer ratio
	\begin{equation}
		R_\eta(k)
		=
		\frac{\|\Pi_{\A,h}\Delta w_h\|_{L^2(\Omega)}}{\|\Delta w_h\|_{L^2(\Omega)}} ,
		\label{eq:delta_transfer_ratio_exp}
	\end{equation}
	where $\Delta w_h$ is the perturbation induced change in the intermediate correction field. The numerator is the component that enters the final weighted gradient recovery space, so $R_\eta$ measures the recoverable fraction of the perturbation correction field.
	
	Table~\ref{tab:perturbation_gain_reference} reports selected propagation gains and the $H^1$ reference response error for $N_{\rm side}=30$. For $k=4$, the gains of the different methods are close. At higher frequencies, the NDM values of $G_{\rm strip}$, $G_{\rm int}$, and $G_{\rm TV}$ generally decrease. Thus the perturbation test identifies a resolution-dependent projection effect: medium and high frequency boundary components are strongly filtered by the recovery space, while low frequency components are transferred comparably to the reference methods. For $k=16$,
	\[
	G_{\rm strip}^{\rm NDM}=1.285\times10^{-2},
	\qquad
	G_{\rm int}^{\rm NDM}=1.423\times10^{-3},
	\qquad
	G_{\rm TV}^{\rm NDM}=6.152\times10^{-1},
	\]
	and $R_\eta$ decreases from $9.580\times10^{-1}$ to $1.056\times10^{-1}$. Thus, at the selected recovery resolution, a smaller part of the intermediate correction field reaches the final weighted gradient recovery space for the tested medium and high frequency perturbations. The curl step transfers the tangential mismatch into an intermediate correction field, while the final projection passes the component representable in $G_{\A,h}$ and attenuates poorly resolved high frequency content. At $k=16$, this filtering is accompanied by the NDM value $E_{\rm ref,H^1}=9.809\times10^{-1}$, larger than those of strong pointwise enforcement and the post hoc $H^1$ optimum penalty candidate. The comparison therefore separates recoverable projected content from fidelity to a deliberately high frequency boundary response at the selected resolution.
	
	\begin{table}[!htbp]
		\centering
		\caption{Incremental boundary perturbation response for \eqref{eq:perturbation_base_problem} to \eqref{eq:boundary_perturbation_mode}, with errors relative to the harmonic reference \eqref{eq:harmonic_reference_increment} to \eqref{eq:reference_relative_error}; $N_{\rm side}=30$ and $\varepsilon=10^{-2}$.}
		\label{tab:perturbation_gain_reference}
		\scriptsize
		\renewcommand{\arraystretch}{0.78}
		\begin{widetable}
			\begin{tabular}{clccccc}
				\toprule
				$k$ & Method & $G_{\rm strip}$ & $G_{\rm int}$ & $G_{\rm TV}$ & $E_{\rm ref,H^1}$ & $R_\eta$ \\
				\midrule
				4
				& Strong pointwise & $2.809{\times}10^{-1}$ & $1.002{\times}10^{-1}$ & $2.838$ & $6.224{\times}10^{-2}$ & n.a. \\
				4
				& Penalty post hoc $H^1$ optimum & $2.771{\times}10^{-1}$ & $9.920{\times}10^{-2}$ & $2.795$ & $6.498{\times}10^{-2}$ & n.a. \\
				4
				& Nitsche trace estimate & $2.790{\times}10^{-1}$ & $1.001{\times}10^{-1}$ & $2.796$ & $1.272{\times}10^{-1}$ & n.a. \\
				4
				& NDM & $2.750{\times}10^{-1}$ & $9.905{\times}10^{-2}$ & $2.649$ & $1.903{\times}10^{-1}$ & $9.580{\times}10^{-1}$ \\
				\midrule
				8
				& Strong pointwise & $2.093{\times}10^{-1}$ & $1.643{\times}10^{-2}$ & $2.466$ & $1.354{\times}10^{-1}$ & n.a. \\
				8
				& Penalty post hoc $H^1$ optimum & $2.034{\times}10^{-1}$ & $1.602{\times}10^{-2}$ & $2.399$ & $1.397{\times}10^{-1}$ & n.a. \\
				8
				& Nitsche trace estimate & $2.020{\times}10^{-1}$ & $1.665{\times}10^{-2}$ & $2.298$ & $2.305{\times}10^{-1}$ & n.a. \\
				8
				& NDM & $1.751{\times}10^{-1}$ & $1.298{\times}10^{-2}$ & $2.112$ & $4.904{\times}10^{-1}$ & $7.734{\times}10^{-1}$ \\
				\midrule
				12
				& Strong pointwise & $1.791{\times}10^{-1}$ & $4.456{\times}10^{-3}$ & $1.836$ & $2.336{\times}10^{-1}$ & n.a. \\
				12
				& Penalty post hoc $H^1$ optimum & $1.713{\times}10^{-1}$ & $4.306{\times}10^{-3}$ & $1.756$ & $2.375{\times}10^{-1}$ & n.a. \\
				12
				& Nitsche trace estimate & $1.622{\times}10^{-1}$ & $3.239{\times}10^{-3}$ & $1.407$ & $3.435{\times}10^{-1}$ & n.a. \\
				12
				& NDM & $7.780{\times}10^{-2}$ & $3.404{\times}10^{-3}$ & $1.910$ & $8.221{\times}10^{-1}$ & $4.187{\times}10^{-1}$ \\
				\midrule
				16
				& Strong pointwise & $1.657{\times}10^{-1}$ & $6.120{\times}10^{-3}$ & $1.511$ & $3.415{\times}10^{-1}$ & n.a. \\
				16
				& Penalty post hoc $H^1$ optimum & $1.556{\times}10^{-1}$ & $5.972{\times}10^{-3}$ & $1.420$ & $3.434{\times}10^{-1}$ & n.a. \\
				16
				& Nitsche trace estimate & $1.356{\times}10^{-1}$ & $8.63{\times}10^{-4}$ & $7.392{\times}10^{-1}$ & $4.539{\times}10^{-1}$ & n.a. \\
				16
				& NDM & $1.285{\times}10^{-2}$ & $1.423{\times}10^{-3}$ & $6.152{\times}10^{-1}$ & $9.809{\times}10^{-1}$ & $1.056{\times}10^{-1}$ \\
				\bottomrule
			\end{tabular}
		\end{widetable}
	\end{table}
	
	\begin{figure}[!htbp]
		\centering
		\includegraphics[width=0.88\textwidth]{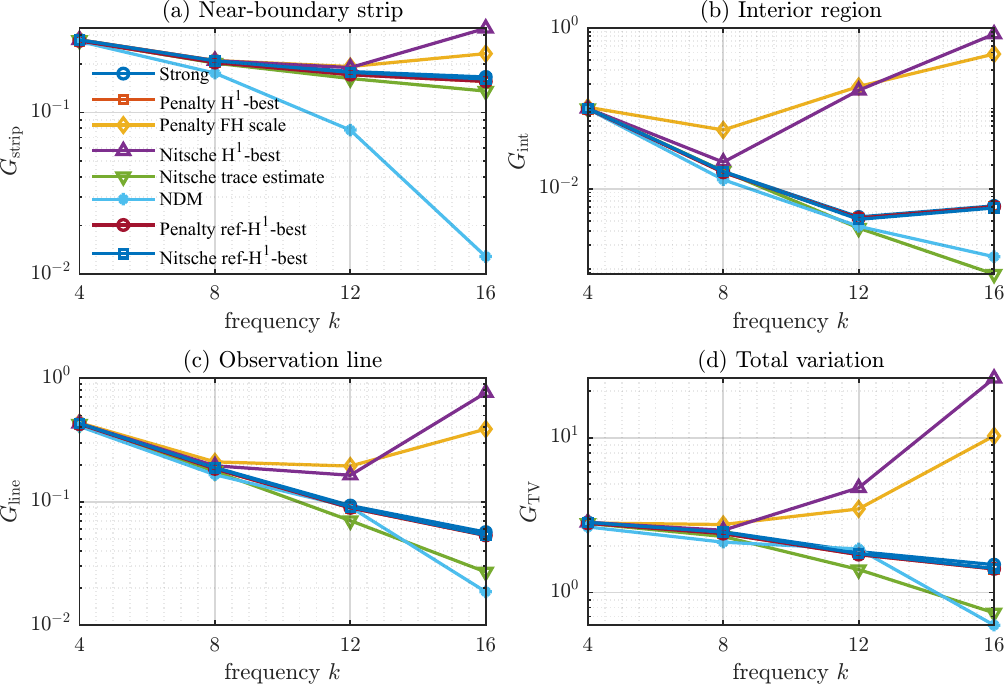}
		\caption{Frequency dependence of the incremental propagation gains $G_{\rm strip}$, $G_{\rm int}$, $G_{\rm line}$, and $G_{\rm TV}$ from \eqref{eq:gains} for the boundary perturbation problem \eqref{eq:perturbed_boundary_data} to \eqref{eq:boundary_perturbation_mode}, with $N_{\rm side}=30$. These gains are read together with the reference fidelity errors in Figure~\ref{fig:perturb_ref_error_frequency}.}
		\label{fig:perturb_gain_frequency}
	\end{figure}
	
	\begin{figure}[!htbp]
		\centering
		\includegraphics[width=0.88\textwidth]{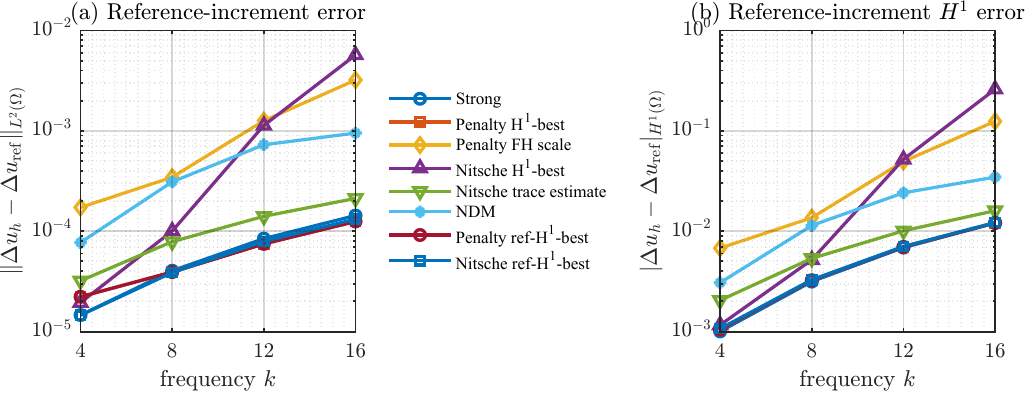}
		\caption{Relative errors of the incremental responses for the perturbation problem \eqref{eq:perturbed_boundary_data} to \eqref{eq:boundary_perturbation_mode}, measured against the harmonic reference \eqref{eq:harmonic_reference_increment} to \eqref{eq:reference_relative_error}. This separates attenuation of boundary noise from fidelity to true essential boundary perturbations.}
		\label{fig:perturb_ref_error_frequency}
	\end{figure}
	
	\FloatBarrier
	
	Figure~\ref{fig:perturb_gain_frequency} shows the frequency dependence of the propagation gains. A smaller gain measures a smaller recovered component of the boundary input; fidelity to a true perturbed problem with essential boundary data is assessed by the harmonic reference in Figure~\ref{fig:perturb_ref_error_frequency}. At $k=16$, the selected recovery resolution yields smaller recovered NDM perturbation components together with a larger harmonic reference error than strong pointwise enforcement and the post hoc $H^1$ optimum penalty candidate. The propagation gains and reference errors therefore distinguish projected attenuation of medium and high frequency boundary components from fidelity to a prescribed high frequency essential boundary response.
	
	The comparisons show a structural distinction. Penalty and Nitsche accuracy is tied to parameter windows, and the post hoc optimum rows use exact solution information unavailable when boundary parameters must be selected a priori. The Lagrange multiplier method introduces sensitivity to the constraint density and saddle point structure. NDM uses the same natural transfer mechanism across these tests, while the perturbation study distinguishes projection attenuation from fidelity to prescribed essential boundary perturbations.
	
	
	\FloatBarrier
	
	\section{Conclusion}
	\label{sec:conclusion}
	
	This paper establishes a natural decomposition method for imposing essential boundary conditions in weak form meshfree Galerkin discretizations with noninterpolatory trial spaces. The difficulty is structural: essential boundary data prescribe a continuous trace, whereas meshfree coefficients or finitely many boundary point values do not generally define the trace admissible space required by the variational problem. NDM avoids this mismatch by introducing the boundary data before discretization through a natural transfer mechanism.
	
	At the continuous level, the method identifies the missing weighted gradient component as a weighted curl range contribution in the topologically trivial single domain setting, yielding an equivalent reconstruction after the boundary mean is fixed. At the discrete level, the projected error decomposition separates the approximation defect of the final recovery space from the upstream transfer error visible to that space. The interface formulation follows the same source, curl, and recovery logic, with a conditional equivalence under the corresponding broken range and lifting assumptions.
	
	The resulting boundary treatment does not require a penalty parameter, a Nitsche stabilization constant, a multiplier space, or a boundary interpolatory modification of the meshfree trial space. The numerical results in a global MQ RBF realization demonstrate the effectiveness of the fixed transfer without boundary parameter tuning and clarify the associated conditioning, computational cost, and perturbation transfer behavior.
	
	Because the continuous transfer is formulated before a particular meshfree basis is chosen, it can be paired with locally supported RBF or RKPM spaces, provided compatible scalar and vector spaces and quadrature are available. Such localized realizations are also the natural route to reduce the large condition numbers inherited from global MQ bases. Future work will therefore focus on sparse realizations for large scale computation, effective preconditioning of the curl correction, and extensions to mixed boundary conditions, more general interface conditions, broader PDE systems, and domains with nontrivial topology.
	
	\section*{Declaration of competing interests}
	
	The authors declare that they have no known competing financial interests or personal relationships that could have appeared to influence the work reported in this paper.
	
	\section*{Data availability}
	
	Data will be made available on request.
	
	\section*{Acknowledgements}
	
	This work was partially supported by the National Natural Science Foundation of China (NSFC) under grant numbers 92370205, 12271512 and 12371377. T. Li was also partially supported by the Jiangsu Provincial Scientific Research Center of Applied Mathematics under Grant No. BK20233002. This research was funded partially by Shanghai Institute for Mathematics and Interdisciplinary Sciences under grant number SIMIS-ID-2024-LG. We thank Tianhe-2 and the Big Data Computing Center in Southeast University, China, for the use of their computing resources.
	
	\section*{Declaration of generative AI and AI-assisted technologies in the writing process}
	
	During the preparation of this work the authors used ChatGPT in order to improve the clarity, fluency, and conciseness of the language. After using this tool/service, the authors reviewed and edited the content as needed and take full responsibility for the content of the published article.
	
	\section*{CRediT authorship contribution statement}
	
	Jingkai Zhang: Software, Validation, Formal analysis, Investigation, Writing: original draft, Visualization.
	Tiexiang Li: Validation, Supervision, Writing: review and editing.
	Shuo Zhang: Conceptualization, Methodology, Supervision, Formal analysis, Validation, Writing: review and editing.

	
	\appendix
	
	\makeatletter
	\renewcommand{\@seccntformat}[1]{%
		\ifstrequal{#1}{section}%
		{Appendix~\csname the#1\endcsname.\hskip 0.5em}%
		{\csname the#1\endcsname.\hskip 0.5em}%
	}
	\@addtoreset{equation}{section}
	\@addtoreset{figure}{section}
	\@addtoreset{table}{section}
	\@ifundefined{c@theorem}{}{\@addtoreset{theorem}{section}}
	\makeatother
	
	\renewcommand{\theequation}{\thesection.\arabic{equation}}
	\renewcommand{\thefigure}{\thesection.\arabic{figure}}
	\renewcommand{\thetable}{\thesection.\arabic{table}}
	\renewcommand{\thetheorem}{\thesection.\arabic{theorem}}
	
	\section{Planar Natural Decomposition}
	\label{app:2d}
	
	For reference, we recall the planar natural decomposition that motivates the three dimensional construction. Let $\Omega\subset\mathbb R^2$ have boundary $\Gamma$ and outward normal vector $\nn=(n_1,n_2)^\top$, and define the tangent vector
	\[
	\mathbf t=(-n_2,n_1)^\top .
	\]
	
	For the Poisson problem
	\[
	-\Delta u=f\quad \text{in }\Omega,
	\qquad
	u=g\quad \text{on }\Gamma,
	\]
	the solution is reconstructed by three sequential subproblems.
	
	First, find $\widetilde u\in H^1_\Gamma(\Omega)$ such that
	\[
	(\nabla\widetilde u,\nabla v)=\langle f,v\rangle,
	\qquad
	\forall v\in H^1_\Gamma(\Omega),
	\qquad
	H^1_\Gamma(\Omega)
	=
	\left\{
	v\in H^1(\Omega):
	\frac{1}{|\Gamma|}\int_\Gamma v\,ds=0
	\right\}.
	\]
	
	Second, find a scalar potential $\varphi$ such that
	\[
	(\curl\varphi,\curl\psi)
	=
	\left\langle
	\partial_{\mathbf t}(g-\widetilde u|_\Gamma),\psi|_\Gamma
	\right\rangle_\Gamma,
	\qquad
	\forall \psi\in H^1(\Omega).
	\]
	
	Finally, find $u_c\in H^1(\Omega)$ such that
	\[
	(\nabla u_c,\nabla v)
	=
	(\nabla\widetilde u-\curl\varphi,\nabla v),
	\qquad
	\forall v\in H^1(\Omega),
	\]
	and set
	\[
	u^\star=u_c-C,
	\qquad
	C=\frac{1}{|\Gamma|}\int_\Gamma (u_c-g)\,ds .
	\]
	
	The complete equivalence theory for the two dimensional formulation is given in~\citep{YuZhang2025}.

	\section{Sequential Error Accounting Estimate}
	\label{app:discrete_stability}
	
	This appendix records the sequential error structure used to interpret the RBF realization. The estimate separates the recovery approximation error from the upstream error component that is visible to the final recovery. Mesh dependent convergence rates are governed by the approximation, quadrature, and algebraic stability estimates of the chosen RBF space and implementation.
	
	Let $V_{h,\Gamma}=V_h\cap H^1_\Gamma(\Omega)$, and let $a_{1,h}$ and $a_{3,h}$ denote the discrete bilinear forms associated with Subproblems~(S1) and~(S3), including the same quadrature, scaling, constraint treatment, and nullspace handling used in the computations. Assume that, on the corresponding mean constrained or quotient spaces, there are constants $\alpha_{j,h}>0$ and $M_{j,h}<\infty$ such that
	\begin{equation}
		\alpha_{j,h}\|v_h\|_{H^1(\Omega)}^2
		\le
		a_{j,h}(v_h,v_h),
		\qquad
		|a_{j,h}(u_h,v_h)|
		\le
		M_{j,h}\|u_h\|_{H^1(\Omega)}\|v_h\|_{H^1(\Omega)},
		\label{eq:app_scalar_stability}
	\end{equation}
	for $j=1,3$ and for all admissible discrete functions.
	
	\begin{proposition}
		\label{prop:app_coercive_steps}
		Under \eqref{eq:app_scalar_stability}, the discrete scalar solves in Subproblems~(S1) and~(S3) are stable on their constrained spaces. If the exact scalar solution of one such step is $z$ and the discrete solution is $z_h$, then
		\begin{equation}
			\|z-z_h\|_{H^1(\Omega)}
			\le
			\left(1+\frac{M_{j,h}}{\alpha_{j,h}}\right)
			\inf_{v_h\in V_h}\|z-v_h\|_{H^1(\Omega)}
			+
			\frac{1}{\alpha_{j,h}}\|\ell-\ell_h\|_{V_h'},
			\label{eq:app_cea_with_perturbation}
		\end{equation}
		where $\ell$ and $\ell_h$ denote the continuous and discrete right hand sides restricted to the discrete test space.
	\end{proposition}
	
	\begin{proof}
		The estimate is the standard Strang form of the C\'ea argument for a symmetric coercive problem. Constants are removed by the boundary mean condition or by the quotient space used in the recovery step. Once this has been done, \eqref{eq:app_scalar_stability} gives the stated bound.
	\end{proof}
	
	We next state the sequential projection identities for the ideal exactly integrated Galerkin structure, and then record the additional perturbation terms that enter the implemented computation. Define the discrete correction and recovery spaces
	\[
	\mathcal R_h
	:=
	\left\{\A^{-1}\nabla\times\boldsymbol\psi_h:
	\boldsymbol\psi_h\in\mathbf V_h\right\},
	\qquad
	G_h
	:=
	\left\{\A\nabla v_h:v_h\in V_h\right\},
	\]
	and let $P_{\mathcal R_h}$ and $P_{G_h}$ be the $L^2(\Omega)^3$ orthogonal projections onto these spaces. Let
	\[
	\mathbf r
	:=
	\mathcal R_\A\boldsymbol\phi
	=
	\A^{-1}\nabla\times\boldsymbol\phi
	=
	\A\nabla\widetilde u-\A\nabla u,
	\qquad
	\delta_h:=\A\nabla(\widetilde u-\widetilde u_h).
	\]
	
	\begin{proposition}
		\label{prop:app_sequential_accounting}
		Work in the topologically trivial single domain setting of Assumption~\ref{ass:weighted_helmholtz_closure}. By Proposition~\ref{prop:checkable_hodge_range}, the continuous correction satisfies
		\[
		\A\nabla u=\A\nabla\widetilde u-\mathbf r .
		\]
		
		Assume, for the moment, exact integration and exact linear solves. Let $\widetilde u_h$ be the Galerkin solution of Subproblem~(S1), let $\mathbf r_h$ be the correction field produced by Subproblem~(S2) using the mismatch $g-T\widetilde u_h$, and let $u_h$ be the final recovered solution. Then
		\begin{equation}
			\mathbf r_h
			=
			P_{\mathcal R_h}(\mathbf r-\delta_h),
			\label{eq:app_rh_projection_identity}
		\end{equation}
		and
		\begin{equation}
			\A\nabla u_h
			=
			P_{G_h}(\A\nabla\widetilde u_h-\mathbf r_h).
			\label{eq:app_uh_projection_identity}
		\end{equation}
		
		Consequently,
		\begin{equation}
			\begin{aligned}
				\|\A\nabla(u-u_h)\|_{L^2(\Omega)}
				&\le
				\|(I-P_{G_h})\A\nabla u\|_{L^2(\Omega)} \\
				&\quad+
				\|P_{G_h}(I-P_{\mathcal R_h})(\mathbf r-\delta_h)\|_{L^2(\Omega)} .
			\end{aligned}
			\label{eq:app_sequential_exact_bound}
		\end{equation}
		
		In particular,
		\begin{equation}
			\|\A\nabla(u-u_h)\|_{L^2(\Omega)}
			\le
			I_{3,h}+I_{2,h}+E_{1,h},
			\label{eq:app_three_defect_bound}
		\end{equation}
		where
		\[
		I_{3,h}:=
		\inf_{v_h\in V_h}
		\|\A\nabla(u-v_h)\|_{L^2(\Omega)},
		\]
		\[
		I_{2,h}:=
		\inf_{\boldsymbol\psi_h\in\mathbf V_h}
		\|\mathbf r-\A^{-1}\nabla\times\boldsymbol\psi_h\|_{L^2(\Omega)},
		\]
		and
		\[
		E_{1,h}:=
		\|\A\nabla(\widetilde u-\widetilde u_h)\|_{L^2(\Omega)}.
		\]
		
		If Subproblem~(S1) satisfies the Strang type estimate
		\begin{equation}
			E_{1,h}
			\le
			C_{1,h}
			\inf_{\xi_h\in V_{h,\Gamma}}
			\|\A\nabla(\widetilde u-\xi_h)\|_{L^2(\Omega)}
			+
			\Delta_{1,h},
			\label{eq:app_s1_strang_energy}
		\end{equation}
		then
		\begin{equation}
			\|\A\nabla(u-u_h)\|_{L^2(\Omega)}
			\le
			I_{3,h}+I_{2,h}+C_{1,h}I_{1,h}+\Delta_{1,h},
			\label{eq:app_accounting_ideal}
		\end{equation}
		where
		\[
		I_{1,h}:=
		\inf_{\xi_h\in V_{h,\Gamma}}
		\|\A\nabla(\widetilde u-\xi_h)\|_{L^2(\Omega)}.
		\]
		
		With quadrature, boundary data approximation, curl right hand side perturbations, recovery perturbations, and algebraic residuals, the practical accounting form is
		\begin{equation}
			\|\A\nabla(u-u_h)\|_{L^2(\Omega)}
			\le
			I_{3,h}+I_{2,h}+C_{1,h}I_{1,h}
			+
			\Delta_{1,h}
			+
			\Delta^R_{2,h}
			+
			\Delta_{3,h}
			+
			\varepsilon_{\rm alg,h}.
			\label{eq:app_accounting_practical}
		\end{equation}
		
		Here $\Delta_{1,h}$ collects the quadrature and right hand side perturbations in Subproblem~(S1), $\Delta^R_{2,h}$ denotes the induced correction field perturbation from Subproblem~(S2), including boundary data and curl right hand side approximation errors, $\Delta_{3,h}$ denotes the perturbation in the final recovery projection, and $\varepsilon_{\rm alg,h}$ collects the residual effects of inexact algebraic solves.
	\end{proposition}
	
	\begin{proof}
		For any $\boldsymbol\eta_h=\A^{-1}\nabla\times\boldsymbol\psi_h\in\mathcal R_h$, the continuous correction satisfies
		\[
		(\mathbf r,\boldsymbol\eta_h)_\Omega
		=
		\mathcal F_{g-T\widetilde u}(\boldsymbol\psi_h).
		\]
		
		The discrete correction uses the mismatch $g-T\widetilde u_h$. Since
		\[
		(g-T\widetilde u_h)-(g-T\widetilde u)
		=
		T(\widetilde u-\widetilde u_h),
		\]
		the lifting definition of the right hand side gives
		\[
		\mathcal F_{g-T\widetilde u_h}(\boldsymbol\psi_h)
		=
		\mathcal F_{g-T\widetilde u}(\boldsymbol\psi_h)
		-
		(\A\nabla(\widetilde u-\widetilde u_h),\boldsymbol\eta_h)_\Omega.
		\]
		
		Therefore
		\[
		(\mathbf r_h,\boldsymbol\eta_h)_\Omega
		=
		(\mathbf r-\delta_h,\boldsymbol\eta_h)_\Omega,
		\qquad
		\forall\boldsymbol\eta_h\in\mathcal R_h,
		\]
		which is precisely \eqref{eq:app_rh_projection_identity}. The final scalar recovery gives \eqref{eq:app_uh_projection_identity}. Since
		\[
		\A\nabla u
		=
		\A\nabla\widetilde u-\mathbf r,
		\]
		we have
		\[
		\A\nabla\widetilde u_h-\mathbf r_h
		=
		\A\nabla u+(I-P_{\mathcal R_h})(\mathbf r-\delta_h).
		\]
		
		Consequently,
		\[
		\A\nabla u-\A\nabla u_h
		=
		(I-P_{G_h})\A\nabla u
		-
		P_{G_h}(I-P_{\mathcal R_h})(\mathbf r-\delta_h),
		\]
		and \eqref{eq:app_sequential_exact_bound} follows by the triangle inequality. Since orthogonal projections are contractions,
		\[
		\|P_{G_h}(I-P_{\mathcal R_h})(\mathbf r-\delta_h)\|_{L^2(\Omega)}
		\le
		\|(I-P_{\mathcal R_h})\mathbf r\|_{L^2(\Omega)}
		+
		\|\delta_h\|_{L^2(\Omega)} .
		\]
		
		This yields \eqref{eq:app_three_defect_bound}. Inserting \eqref{eq:app_s1_strang_energy} gives \eqref{eq:app_accounting_ideal}. The perturbation terms in \eqref{eq:app_accounting_practical} are obtained by applying the same stability argument to the quadrature, right hand side, boundary data, recovery, and algebraic perturbations in Subproblems~(S1), (S2), and~(S3).
	\end{proof}
	
	Proposition~\ref{prop:app_sequential_accounting} explains the diagnostics used in Section~\ref{sec:exp}. The sharper diagnostic in the numerical section is the projected split in \eqref{eq:numerical_error_split}: the final weighted gradient error is interpreted through the recovery approximation component and the upstream component visible to the recovery space. The coarser bound \eqref{eq:app_three_defect_bound} supplies the sequential error accounting, while implementation specific convergence rates are governed by approximation estimates for the selected RBF space, mesh dependent stability bounds, quadrature estimates, and algebraic residual bounds.

	\section{Computational Cost of the Dense Global RBF Realization}
	\label{app:validation_scale_cost}
	
	This appendix records the algebraic sizes, conditioning estimates, and solver choices used in the reported three dimensional tests. The implementation is based on dense global MQ RBF spaces for Subproblems 1 and 3, together with a matrix-free MINRES QLP solve for the semidefinite curl curl system in Subproblem 2.
	
	\par\smallskip
	\phantomsection
	\refstepcounter{table}
	\label{tab:cost_3d_examples}
	\noindent\textbf{Table~\thetable}\par
	\noindent Computational cost at the largest reported node set for the three dimensional tests \eqref{eq:3d_model_problem}, \eqref{eq:case5_piecewise_solution} to \eqref{eq:case5_jump_data}, and \eqref{eq:case6_piecewise_solution} to \eqref{eq:case6_kappa2}.
	\par\smallskip
	\begin{center}
		\begingroup
		\scriptsize
		\setlength{\tabcolsep}{2.7pt}
		\renewcommand{\arraystretch}{1.08}
		\resizebox{\textwidth}{!}{%
			\begin{tabular}{c l c c c c c c c c}
				\toprule
				Case
				& Type
				& $N_{\rm side}$
				& DOF $(S1/S2/S3)$
				& Linear systems
				& $\mathrm{cond}_{\max}^{\mathrm{est}}$
				& CPU time (s)
				& Peak memory (GB)
				& Solver $(S1/S2/S3)$
				& S2 iter. \\
				\midrule
				
				1
				& Poisson
				& 21
				& $9266/27794/9280$
				& $3$
				& $5.915{\times}10^{17}$
				& $4825.37$
				& $9.06$
				& direct / MINRES QLP / direct
				& $1297$ \\
				
				2
				& Nonpolynomial Poisson
				& 21
				& $9266/27794/9280$
				& $3$
				& $5.915{\times}10^{17}$
				& $4866.37$
				& $9.06$
				& direct / MINRES QLP / direct
				& $1304$ \\
				
				3
				& Smooth variable coefficient
				& 21
				& $9266/27794/9280$
				& $3$
				& $8.262{\times}10^{17}$
				& $4373.56$
				& $9.05$
				& direct / MINRES QLP / direct
				& $934$ \\
				
				4
				& Nonsmooth variable coefficient
				& 21
				& $9266/27794/9280$
				& $3$
				& $9.118{\times}10^{17}$
				& $3970.77$
				& $3.99$
				& direct / MINRES QLP / direct
				& $447$ \\
				
				5
				& Flat interface jump
				& 21
				& $9266/27809/9740$
				& $4$
				& $4.895{\times}10^{21}$
				& $2612.13$
				& $5.38$
				& direct / MINRES QLP / direct$\times 2$
				& $899$ \\
				
				6
				& Internal interface with nonconstant $\kappa_1$
				& 19
				& $6870/20588/7770$
				& $4$
				& $3.188{\times}10^{21}$
				& $2080.72$
				& $0.38$
				& direct / MINRES QLP / direct$\times 2$
				& $686$ \\
				
				\bottomrule
			\end{tabular}%
		}
		\endgroup
	\end{center}
	\par\smallskip
	
	\medskip
	Table~\ref{tab:cost_3d_examples} reports the cost of one sequential NDM solve at the largest node set used for each three dimensional test. The CPU time and peak memory exclude auxiliary diagnostics such as $\Eapp$, $\Eproj$, $T_{123}$, and $T_{23}$. Subproblems 1 and 3 are solved by direct linear algebra. Subproblem 2 is solved by matrix-free MINRES QLP, and the reported Krylov iteration count therefore refers only to this curl correction solve. For interface tests, the recovery step is solved separately on the two subdomains, denoted by direct$\times 2$.
	
	The condition estimate column is defined by
	\[
	\mathrm{cond}_{\max}^{\mathrm{est}}
	=
	\max\{
	\mathrm{cond}_{S1}^{\mathrm{rcond}},
	\mathrm{cond}_{S2}^{\mathrm{pc}},
	\mathrm{cond}_{S3}^{\mathrm{rcond}}
	\}.
	\]
	
	Here $\mathrm{cond}_{S1}^{\mathrm{rcond}}$ and $\mathrm{cond}_{S3}^{\mathrm{rcond}}$ are reciprocal condition estimates for the direct systems after the same basis scaling, constraint treatment, and nullspace handling used in the computation. The quantity $\mathrm{cond}_{S2}^{\mathrm{pc}}$ is the MINRES QLP estimate for the preconditioned matrix-free curl curl operator. For interface tests with piecewise recovery, $\mathrm{cond}_{S3}^{\mathrm{rcond}}$ is the maximum over the two subdomain recovery systems.
	
	The first five rows use cached double precision derivative blocks in the matrix-free Subproblem 2 operator. The internal cube interface test in Case 6 uses a different storage strategy: the derivative blocks are handled in batches rather than stored as one full cached set. This batching reduces peak memory substantially, as reflected by the last row of Table~\ref{tab:cost_3d_examples}. The reported CPU time includes the cost of this storage strategy. Since Case 6 also uses a smaller stable node set, $N_{\rm side}=19$, its memory and timing are interpreted together with both the algebraic size and the batching mode.
	
	The table records the cost of the dense global RBF implementation. Single domain tests require three linear systems, while interface tests require four because the final recovery is carried out separately on the two subdomains. This cost is distinct from the cost of boundary parameter selection in penalty or Nitsche methods; the comparison in Section~\ref{subsec:ebc_mechanism_tests} treats these two issues separately.

	\bibliographystyle{unsrtnat}
	\bibliography{references}
	
\end{document}